\documentclass[final]{cmslatex}
\usepackage[paperwidth=7in,paperheight=10in,margin=.875in]{geometry}
\usepackage{amsfonts,amssymb}
\usepackage{amsmath}
\usepackage{mathtools}
\usepackage{graphicx}
\graphicspath{{./fig/}}
\usepackage{epstopdf}
\usepackage{cite}
\usepackage{enumerate}
\usepackage{url}
\usepackage{float}
\usepackage{subcaption}
\makeatletter
\newcommand{\FloatBarrier}{%
  \par
  \begingroup
  \ifx\@deferlist\@empty
    \ifx\@botlist\@empty
    \else
      \clearpage
    \fi
  \else
    \clearpage
  \fi
  \endgroup
}
\makeatother
\usepackage[hidelinks]{hyperref}
\newcommand{\doi}[1]{\href{https://doi.org/#1}{\texttt{\detokenize{https://doi.org/#1}}}}
\allowdisplaybreaks
\renewcommand{\theequation}{\arabic{section}.\arabic{equation}}
\newcommand{\T}{\top}
\newcommand{\F}{\mathrm{F}}

\newcommand{\E}{\mathbb{E}}
\newcommand{\tr}{\mathrm{tr}}
\begin{document}

\title{A stochastic gradient algorithm for non-separable optimization with convergence guarantee}

\author{Yingzhou Li\thanks{School of Mathematical Sciences, Fudan University, Shanghai 200433, China; Shanghai Key Laboratory for Contemporary Applied Mathematics, Fudan University, Shanghai 200433, China (\nolinkurl{yingzhouli@fudan.edu.cn}). ORCID: 0000-0003-1852-3750.}
\and Ruofan Wu\thanks{School of Mathematical Sciences, Fudan University, Shanghai 200433, China (\nolinkurl{rfwu22@m.fudan.edu.cn}). Corresponding author. ORCID: 0009-0004-7589-5280.}}

\pagestyle{myheadings}
\markboth{A STOCHASTIC GRADIENT ALGORITHM}{Y. LI AND R. WU}
\maketitle

\begin{abstract}
    We study smooth parameterized non-separable objectives of the form $\mathcal L(\theta)=F(Y(\theta))$, where a global operator couples the outputs of all training samples. We analyze two stochastic gradient constructions within this output-space framework: the ideal per-batch gradient $G$, obtained by row sampling of the full output gradient, and a cached surrogate $H$ for parameterized updates in which full-data quantities are expensive to evaluate.
    For uniformly sampled batches, the ideal estimator has a deterministic batch-size scaling and a second moment controlled by the full output gradient. Using this structure, we establish local high-probability confinement and convergence for fixed positive step sizes on a prescribed parameter neighborhood under a power-type gradient inequality, with linear and polynomial regimes determined by its exponent.
    For the cached surrogate, a Lyapunov function couples the objective gap with the cache error and yields analogous local results on the corresponding non-exit event. The theory includes explicit path-length, confinement, and rate bounds; numerical experiments illustrate finite-budget decay trends and the effects of global output-gradient sampling, batch restriction, and cache tracking.
\end{abstract}


\section{Introduction}
\par
\label{intro}
Modern learning objectives often couple the outputs of all training samples through a global operator. We study this setting as a class of smooth parameterized non-separable optimization problems, written as
\[
\mathcal L(\theta)=F(Y(\theta)),
\]
where $Y(\theta)$ collects the sample outputs and $F$ acts on the full output matrix. We develop a local convergence theory for stochastic updates in this class.
The standard setting is expressed as a sample-wise separable finite sum,
\[
\mathcal{L}(\theta) = \sum_{i=1}^{N} \ell(\theta; x_i, y_i),
\]
where each term depends only on one sample and mini-batch gradients inherit the corresponding finite-sum structure. In contrast, the objectives considered here take the coupled form
\[
\mathcal{L}(\theta) = F\left( \{\phi(\theta; x_i, y_i)\}_{i=1}^N \right),
\]
where $F$ depends on the full collection of outputs. A batch-restricted evaluation then changes the operator itself, so its gradient need not estimate the gradient of the full coupled objective. This distinction is structural and is separate from the sampling variance of an estimator formed from the full output gradient.

Training deep neural networks is typically approached through iterative stochastic first-order methods, with SGD and its adaptive variants—Adagrad \cite{duchi2011adaptive}, ADADELTA \cite{zeiler2012adadelta}, Adam \cite{kingma2014adam}—forming the practical backbone. These methods estimate gradients from randomly sampled mini-batches, which drastically reduce per-iteration cost relative to full-batch approaches and make large-scale optimization tractable \cite{duchi2011adaptive,zeiler2012adadelta,kingma2014adam,schmidt2017minimizing,johnson2013accelerating,defazio2014saga}. Variance-reduction methods can accelerate convergence by reducing stochastic gradient noise for finite-sum objectives. When the objective is non-separable, however, a batch-restricted gradient changes the global coupling operator and can introduce a structural discrepancy distinct from ordinary stochastic variance. Specialized methods address particular non-decomposable settings, but their constructions and guarantees do not directly provide a general theory for coupled objectives.

For stochastic gradients with a nonvanishing noise floor, fixed-step analyses generally yield neighborhood-type guarantees. Exact fixed-step convergence requires control of the stochastic error in addition to suitable objective geometry. PL/KL gradient inequalities relate the true gradient to the objective gap but do not by themselves eliminate a nonvanishing stochastic noise floor. Interpolation under suitable regularity assumptions, or a relative second-moment condition, can provide the required vanishing-noise structure. Classical analyses of general stochastic nonconvex optimization often use diminishing step sizes to obtain asymptotic stationarity, while variance-reduction methods can obtain exact convergence for separable finite-sum problems at the cost of additional memory or periodic full-gradient computations \cite{johnson2013accelerating,defazio2014saga}. Modern machine learning (ML) and artificial intelligence (AI) systems increasingly rely on highly over-parameterized models, particularly deep neural networks (DNNs). In this regime, models are sufficiently expressive to represent optimal solutions and can nearly perfectly fit training data even when the number of parameters far exceeds the sample size \cite{zhang2021understanding,tan2019efficientnet,huang2019gpipe,kolesnikov2020big}. Recent theoretical work combines objective geometry with interpolation or related noise-control conditions to support fixed-step convergence in several separable or over-parameterized settings. Bassily et al. \cite{bassily2018exponential} proved exponential convergence rates for SGD on over-parameterized nonconvex models under interpolation. Under a strong-growth condition together with gradient dominance (the Polyak--\L{}ojasiewicz condition), Vaswani et al. \cite{vaswani2019fast} obtained linear convergence rates without requiring strong convexity. Liu et al. \cite{liu2023aiming} studied fast convergence and implicit regularization for SGD in over-parameterized problems. These results motivate the present focus on a different estimator-level issue: non-separability changes the batch operator itself, so fitting a batch-restricted objective does not generally imply minimizing the full coupled objective. 

The relevant issue is whether a batch-restricted objective preserves the gradient structure of the full coupled objective. The non-separable setting therefore requires both a suitable batch-gradient construction and local conditions under which its stochastic second moment vanishes with the full gradient. In our analysis, row sampling supplies the second-moment control for $G$, while a Lyapunov argument additionally controls the cache error for $H$. Together with the KL geometry, these estimates yield finite-length, local confinement, and convergence-rate guarantees with a positive constant step size. Our main contributions are:

\begin{itemize}
    \item We formulate a general output-space framework for smooth parameterized non-separable objectives and distinguish the ideal row-sampled estimator $G$ from a cached surrogate $H$ for parameterized updates.
    \item We identify the estimator structure that yields a deterministic batch scaling and a full-gradient-controlled second moment, and use it to prove local high-probability confinement and convergence for fixed positive step sizes on the non-exit event.
    \item We derive KL-exponent-dependent linear and polynomial rates. On the corresponding non-exit events, the full-sequence almost-sure bounds admit any slightly weaker geometric factor in the linear case and an arbitrarily small exponent loss in the polynomial case. We extend the analysis to cached surrogates through a Lyapunov function coupling objective and cache errors.
\end{itemize}

The remainder of the paper is organized as follows. In Section 2 we present the algorithm. Section 3 gives a local convergence analysis for the ideal per-batch gradient $G$, and Section 4 presents analogous results for the cached surrogate $H$. Section 5 reports numerical experiments that illustrate the empirical behavior of the method. We conclude with final remarks and directions for future work in Section 6.
\section{A novel stochastic gradient algorithm}
\par

\subsection{Background.}\label{background}
Building on the problem setting introduced in Section~\ref{intro}, this section reviews the mini-batch SGD framework and formalizes the structural bias that arises when global sample coupling prevents unbiased gradient estimation.
Most neural network training problems are optimization problems: one minimizes the loss function $\mathcal{L}(\theta)$ to obtain optimal parameters $\theta^{}$. Classical (full-batch) gradient descent updates parameters by
\begin{equation*}
\theta_{k+1} = \theta_k - \eta \nabla_\theta \mathcal{L}(\theta_k),
\end{equation*}
but this is computationally prohibitive for the large datasets typical in deep learning. To address this, stochastic gradient descent \cite{robbins1951stochastic}, and in particular its mini-batch variant, uses randomly sampled subsets $\mathcal{B}_k$ to approximate the true gradient. The parameter update at step $k$ is
\begin{equation*}
\theta_{k+1} = \theta_{k} - \eta_{k} \frac{1}{|\mathcal{B}_k|}
\sum_{i \in \mathcal{B}_k} \nabla_\theta \ell(\theta_k; x_i, y_i),
\end{equation*}
where $x_i$ and $y_i$ denote the input features and target (or label) of the $i$-th sample, respectively, $\ell(\cdot)$ denotes the per-example loss, $|\mathcal{B}_k|$ is the mini-batch size (typically 32–512), and $\eta_k$ is the learning rate. This approach offers three practical properties \cite{bottou2010large,dean2012large}:
\begin{itemize}
\item \textit{Hardware efficiency}: mini-batches enable parallel processing (e.g., on GPUs or TPUs) and higher hardware utilization;
\item \textit{Variance reduction}: averaging over a batch produces a more stable gradient estimator than single-sample updates;
\item \textit{Implicit regularization}: controlled stochasticity often improves generalization.
\end{itemize}

The standard mini-batch SGD framework assumes a decomposable objective $\mathcal{L}(\theta)=\mathbb{E}_{(x,y)}[\ell(\theta;x,y)]$, so that per-batch gradients provide unbiased estimates of the full gradient. Modern objectives couple the entire dataset through a global operator of the form
\[
\mathcal{L}(\theta) = F\left( \{ \phi(\theta; x_i, y_i) \}_{i=1}^N \right),
\]
where $F$ is a coupling operator (for example, a sorting or a statistic over the dataset). In such settings, the gradient of the batch-restricted objective is generally a biased estimator of $\nabla_\theta\mathcal{L}$, with the bias depending on $|\mathcal{B}_k|$ and on the sensitivity of $F$ \cite{zaheer2017deep}.

This issue is particularly acute in unsupervised learning, where objectives inherently 
involve relational coupling between samples \cite{dean2012large,keskar2017on}. 
Given unlabeled data $\mathcal{X} = \{x_i\}_{i=1}^N$, 
losses typically take the form:
\[
\mathcal{L}(\theta) = F\Big( \{ \phi_\theta(x_i), \{ \phi_\theta(x_j) \}_{j \in \mathcal{R}(x_i)} \}_{i=1}^N \Big)
\]
where $\phi_\theta$ is a representation encoder and $\mathcal{R}(x_i)$ denotes a relational set (for example, neighbors used in contrastive learning). The corresponding mini-batch gradient commonly used in practice is
\[
\tilde{g}_k = \nabla_\theta F\left( \{ \phi_\theta(x_j), \{ \phi_\theta(x_l) \}_{l \in \mathcal{R}(x_j) \cap \mathcal{B}_k} \}_{j \in \mathcal{B}_k} \right)
\]
which gives rise to two challenges beyond the usual stochastic gradient variance \cite{ruder2016overview}:
\begin{itemize}
\item \textit{Relation truncation}: the intersection $\mathcal{R}(x_j)\cap\mathcal{B}_k$ may omit important long-range dependencies, degrading the fidelity of the batch-wise objective;
\item \textit{Biased statistics}: dataset-level operators $F$ (e.g., distribution-matching or global moments) produce skewed statistics when evaluated on small batches, resulting in biased gradient estimates.
\end{itemize}
While reducing the step size can mitigate the effect of stochasticity, it slows convergence. To achieve fast convergence with constant step sizes, we therefore seek stochastic gradient algorithms that produce (or closely approximate) unbiased gradient estimates in the presence of inter-sample coupling. 

\subsection{Setup.}
\label{subsec:setup}
In the following, for convenience and consistency we express the loss as:
\[
\mathcal{L}(\theta) = F\big( Y(\theta) \big)
\]
where $Y(\theta)\in\mathbb{R}^{N\times d}$ is the output matrix whose rows are sample-wise outputs, and $F$ is a nonlinear function. The convergence analysis imposes smoothness and a power-type gradient inequality on a fixed output domain, together with separate bounds on the parameterization.

In traditional mini-batch training one evaluates the model on a subset (batch)
$\mathcal{B}_k$ and uses the resulting batch-based quantity to form an update. To represent this formally we use a masking interpretation: 
let $\text{mask}_{\mathcal{B}}:\mathbb{R}^{N\times d}\rightarrow\mathbb{R}^{N\times d}$ denote the operator that keeps rows indexed by
$\mathcal{B}$ and ignores (i.e., masks out) other rows. We then denote the batch-evaluated object by 
$Y_{\mathcal{B}}\coloneqq \text{mask}_{\mathcal{B}}(Y)$. Note that when $F$ is sample-wise separable (i.e., $F(Y)=\sum_{n}f(Y_n)$) this mask is equivalent to zeroing out the non-batch rows; however, when 
$F$ depends on global statistics (e.g., $Y^{\T}Y$ or distribution-matching operators) zeroing out the non-batch rows is not equivalent to true batch evaluation. 
In such cases, we treat the batch operator as ``apply $F$ only to the submatrix formed by the batch rows'' (i.e., evaluate $F$ on the dataset restricted to $\mathcal{B}$), and denote this by $F_{\mathcal{B}}(Y)$.

By the chain rule the full-parameter gradient can be written compactly using a contraction over the sample and output indices. For notational clarity, we denote this contraction by $\times$. Concretely, for an $N\times d$ matrix $A$ and an $N\times d\times s$ tensor $B$, we
define
\[
A\times B\in\mathbb{R}^{s},\quad (A\times B)_{m} = \sum_{n=1}^{N}\sum_{j=1}^{d} A_{nj}B_{njm},
\]
so that
\begin{equation*}
    \nabla_\theta \mathcal{L} = \nabla_{Y}F(Y(\theta))\times\nabla_{\theta}Y(\theta).
\end{equation*}
Equivalently, writing per-sample Jacobians $\nabla_{\theta}Y_{n}(\theta)\in\mathbb{R}^{d\times s}$ and the 
$n$-th row $\nabla_{Y}F_{n}\in\mathbb{R}^{d}$, we have the componentwise expansion
\begin{equation*}
\nabla_\theta \mathcal{L} = \sum_{n=1}^{N} (\nabla_{Y} F)_n \cdot \nabla_{\theta} Y_n (\theta)
\end{equation*}
where the dot denotes multiplication of a row vector with a 
$d\times s$ Jacobian yielding an 
$s$-vector.

In the mini-batch formulation, it is useful to distinguish the zero-masked full-space expression from the genuinely batch-restricted objective. If $P_{\mathcal B}$ denotes the row projection that keeps the rows in $\mathcal B$ and zeros the others, then the chain rule gives
\[
\nabla_\theta F(P_{\mathcal B}Y(\theta))
=J(\theta)^\top P_{\mathcal B}^{\top}
\operatorname{vec}\!\left(\nabla_YF(P_{\mathcal B}Y(\theta))\right),
\]
where $J(\theta)$ is the vectorized full-data Jacobian. This zero-masked expression is not, in general, the same as evaluating a coupled objective on the submatrix formed by $\mathcal B$. For the latter operation, which is the one used by a practical batch-restricted method, we write
\[
\widetilde g_k
=J(\theta_k)^\top\operatorname{vec}\!\left(\nabla_YF_{\mathcal B_k}(Y(\theta_k))\right).
\]
The derivative of $F_{\mathcal B_k}$ is understood with respect to the full output variable, with zero rows outside $\mathcal B_k$. Equivalently, the contraction is over the per-sample Jacobians for $n\in\mathcal B_k$. Thus the structural bias discussed here is induced by the batch operator.
For a separable objective $F(Y)=\sum_i f_i(Y_i)$ with the unnormalized batch-sum convention, $\mathbb E_{\mathcal B}[\nabla_YF_{\mathcal B}(Y)]=p\nabla_YF(Y)$, where $p=b/N$. For a coupled objective, the additional output-space discrepancy is therefore measured by
\[
\E_{\mathcal{B}}\!\left[\operatorname{vec}(\nabla_YF_{\mathcal{B}}(Y))\right]-p\operatorname{vec}(\nabla_YF(Y)),
\]
while the actual parameter-gradient discrepancy relative to the ideal row-sampled estimator is
\[
\E_{\mathcal{B}}[\widetilde g_k\mid\mathcal F_k]-p\nabla_\theta\mathcal L(\theta_k)
=J(\theta_k)^\top\!\left(
\E_{\mathcal{B}}\!\left[\operatorname{vec}(\nabla_YF_{\mathcal{B}}(Y(\theta_k)))\right]
-p\operatorname{vec}(\nabla_YF(Y(\theta_k)))\right).
\]
Thus the usual factor $p$ should be distinguished from the additional bias caused by restricting sample interactions. The output-space discrepancy need not produce a parameter-gradient discrepancy if it lies in the nullspace of $J(\theta_k)^\top$. In practical implementations, automatic differentiation computes the per-sample Jacobian blocks for the current batch exactly; the structural bias arises from the batch operator $F_{\mathcal B}$, not from those deterministic Jacobian blocks.
Revisiting the example from Section~\ref{background}, consider
\[
\mathcal{L}(\theta) = F\Big( \{ \phi_\theta(x_i), \{ \phi_\theta(x_j) \}_{j \in \mathcal{R}(x_i)} \}_{i=1}^N \Big)
\]
where $\phi_\theta$ is a representation encoder and $\mathcal{R}(x_i)$ denotes a relational set.
The practical mini-batch evaluation replaces the global interactions by their batch-restricted counterparts, leading to a batch gradient of the form
\[
\tilde{g}_k = \nabla_Y F\left( \{ \phi_\theta(x_j), \{ \phi_\theta(x_l) \}_{l \in \mathcal{R}(x_j) \cap \mathcal{B}_k} \}_{j \in \mathcal{B}_k} \right)\times \nabla_{\theta}Y(\theta_k),
\]
with the output-gradient factor alone carrying the relation truncation and skewed statistics.

\subsection{Ideal and Surrogate Stochastic Gradient Descent.}
\label{subsec:grad}
The ideal per-batch gradient, denoted $G$, selects rows of the full output gradient and contracts them with the per-sample Jacobians. Its conditional mean has a deterministic batch-size factor, while its second moment is controlled by the full output gradient. The cached surrogate $H$ replaces expensive global terms by quantities maintained from cached outputs. The distinction is substantive for parameterized maps $Y(\theta)$: a parameter update can change outputs outside the current batch, so the cache can become stale even when only batch rows are used to form the update.

Define the ideal batch-evaluated gradient (using the masked full output) as 
\[
G= \nabla_{\mathcal{B}} F
\left( Y(\theta) \right)\times 
\nabla_{\theta} Y(\theta),
\]
where $\nabla_{\mathcal{B}} F$ is the $N \times d$ matrix that equals $\nabla_{Y} F$ 
on rows indexed by the batch $\mathcal{B}$ and is zero elsewhere.

Since $\nabla_\theta Y$ can be viewed as the vertical concatenation of per-sample Jacobians,
\[
\nabla_{\theta} Y = \begin{bmatrix}
    \nabla_{\theta} Y_1 \\
    \nabla_{\theta} Y_2 \\
    \vdots \\
    \nabla_{\theta} Y_N
\end{bmatrix}, \quad 
\nabla_{\theta} Y_n = \frac{\partial Y_n}{\partial \theta} \in \mathbb{R}^{d \times s}.
\]
Here $\nabla_{\theta} Y_n$ is the Jacobian of the $n$-th sample's output $Y_n \in \mathbb{R}^d$ 
with respect to parameters $\theta$. Thus the parameter gradient can be written as
\[
\nabla_\theta \mathcal{L} = \nabla_{Y} F(Y(\theta)) \times \nabla_{\theta} Y
= \sum_{n=1}^{N} (\nabla_{Y} F)_n \cdot \nabla_{\theta} Y_n (\theta)
\]
where $(\nabla_{Y} F)_n$ is the $n$-th row of $\nabla_{Y} F$ and 
$\nabla_{\theta} Y_n (\theta)$ is the Jacobian of the $n$-th output. Therefore,
\[
G = \nabla_{\mathcal{B}} F
\left( Y(\theta) \right)\times 
\nabla_{\theta} Y(\theta) = \sum_{n\in\mathcal{B}} (\nabla_{Y} F)_n \cdot \nabla_{\theta} Y_n (\theta)
\] 
which matches the computational pattern of standard mini-batch gradients. For a uniformly sampled size-$b$ subset, with $p=b/N$, the conditional mean of $G$ is $p$ times the full parameter gradient. Thus $G/p$ is unbiased; we retain the unnormalized definition of $G$ and include this deterministic factor in the convergence constants. Moreover, the conditional second-moment estimate proved below gives
\[
\mathbb E[\|G_k/p\|^2\mid\mathcal F_k]
\leq \frac{B_1^2}{p}\|g_k\|_{\mathrm F}^2
\leq \frac{B_1^2}{p\lambda_{\min}}
\|\nabla_\theta\mathcal L(\theta_k)\|^2,
\]
where the second inequality uses A5. Thus the row-sampled estimator has a strong-growth- or relative-second-moment-type control: its stochastic second moment vanishes with the full gradient. The control is induced by row sampling of a common full output gradient and the full-data Jacobian condition A5 for a non-separable objective, and it supports constant-step local convergence. The KL condition enters through the subsequent finite-length, confinement, and rate arguments. For the cached estimator $H$, cache error leads to a Lyapunov analysis coupling objective descent and cache tracking. Each batch is sampled without replacement within the batch, independently across iterations conditional on the past; the same index may therefore occur in batches from different iterations.

For a direct output-space block update that changes only the selected rows, exact cache initialization and post-update refreshing make the cache exact and hence identify $H$ with $G$. The cached-surrogate analysis below concerns the parameterized setting, where changing $\theta$ can modify unselected rows and produce a nonzero cache error. To construct such a surrogate, we introduce an approximate output gradient $\widetilde{\nabla_{Y} F}$ that depends on a cached approximation $\tilde{Y}$ and on the current $Y$:
$$\widetilde{\nabla_{Y} F} = \widetilde{\nabla_{Y} F}(\tilde{Y},Y),$$ 
and define the practical batch estimator
\[
H = \widetilde{\nabla_{\mathcal{B}}F}
\left( Y(\theta) \right)\times 
\nabla_{\theta} Y(\theta) = \sum_{n\in\mathcal{B}} (\widetilde{\nabla_{Y} F})_n \cdot \nabla_{\theta} Y_n (\theta),
\]
where $\widetilde{\nabla_{\mathcal{B}}F}$ equals $\widetilde{\nabla_{Y}F}$ on rows in $\mathcal{B}$ and is zero elsewhere.
The approximation $\widetilde{\nabla_{Y} F}(\tilde{Y},Y)$ is constructed by evaluating computationally inexpensive terms at the current
$Y$ while substituting $\tilde{Y}$ for the expensive-to-evaluate global terms.
Concretely:
\begin{itemize}
\item $\tilde{Y}$ is a cached approximation of $Y$  maintained across iterations; on iteration $k$ we update the cached rows for the current batch and keep other rows unchanged:
\[ (\tilde{Y}_{{\text{new}}})_{n} = \left\{ 
     \begin{array}{ll}
        (Y_{\text{new}})_{n} & \text{if }  n \in \mathcal{B} \\
       (\tilde{Y}_{{\text{old}}})_{n} & \text{if } n \notin \mathcal{B}
     \end{array}
   \right. \]
\item The approximation replaces only the expensive parts of $\nabla_Y F$ by expressions evaluated
at $\tilde{Y}$; inexpensive or local terms are evaluated at the current $Y$.
\item This design allows incremental updates of global quantities (e.g., $\tilde{Y}^{\T}\tilde{Y}$) using low-cost rank-$|\mathcal{B}|$ updates. 
For instance, if $\nabla_{Y} F(Y) = AY + BY^{\T}Y$ with A sparse or diagonal, then $AY$ can be evaluated at the current $Y$ while
$BY^{\T}Y$ is approximated using $\tilde{Y}$. The cached Gram matrix admits the cheap update
\[
\tilde{Y}_{\text{new}}^{\top}\tilde{Y}_{\text{new}} = \tilde{Y}_{\text{old}}^{\top}\tilde{Y}_{\text{old}} 
-\sum\limits_{n\in\mathcal{B}}(\tilde{Y}_{\text{old}})_{n}^{\top}(\tilde{Y}_{\text{old}})_{n}
+\sum\limits_{n\in\mathcal{B}}(\tilde{Y}_{\text{new}})_{n}^{\top}(\tilde{Y}_{\text{new}})_{n}.
\]
Hence one can update $\widetilde{\nabla_{\mathcal{B}}F}$ without recomputing quantities over $\mathcal{B}^c$ (more precisely 
over $\mathcal{B}_{\text{new}}^c\cap\mathcal{B}_{\text{old}}^c$).
\end{itemize}

Note that whereas the standard mini-batch gradient is obtained directly by 
automatic differentiation, our proposed gradient requires additional processing. 
Concretely, one can obtain the parameter update by differentiating the surrogate contraction
$$\tr(Y_{\mathcal{B}}(\theta)^{\T}\nabla_{\mathcal{B}}F)\quad
\left(\text{or}\quad\tr(Y_{\mathcal{B}}(\theta)^{\T}\widetilde{\nabla_{\mathcal{B}}F})\right),$$
treating $Y_{\mathcal{B}}(\theta)$ as a function of $\theta$ while holding $\nabla_{\mathcal{B}}F$ (or $\widetilde{\nabla_{\mathcal{B}}F}$) fixed. Its parameter gradient is therefore exactly $G$ (or $H$).

The ideal estimator has a vanishing second moment at a stationary output, and the following sections establish local convergence results for $G$ and for the cached surrogate $H$.
\subsection{Example of the gradient algorithm.}
In this section, we discuss the relationship between the proposed method and existing algorithms. Under specific 
conditions the proposed estimator coincides with prior methods.
\begin{itemize}
    \item Sample-wise separable losses. 
    If the loss decomposes per sample, 
    \[
    \mathcal{L}(\theta) = \sum_{n=1}^N f_n(y_n(\theta)),
    \]
    then the gradient of the loss has separate sample rows, with $(\nabla_YF)_n$ depending only on $Y_n$; if a second derivative is considered, the corresponding Hessian is block diagonal across samples. In this case the parameter gradient 
    decomposes into independent per-sample terms. In this case the mini-batch 
    update equals the ideal per-batch gradient $G$, so the cached 
    surrogate $H$ is unnecessary and the method reduces to standard 
    mini-batch SGD.\\
    \item Relation to SpecNet2 \cite{chen2022specnet2}. SpecNet2 studies a particular spectral-embedding objective and provides an instance of the full-data, batch-restricted, and localized gradient choices represented in our framework. The present formulation abstracts the output-gradient construction to general coupled objectives $F(Y)$ and analyzes its pullback through a parameterized map $Y(\theta)$. In particular, parameter updates can change outputs outside the current batch, making cache drift distinct from the direct output-space block-update case. The Lyapunov and KL analysis below is the additional component needed for the cached surrogate in this parameterized setting.
\end{itemize}
\section{Local convergence analysis for the ideal per-batch gradient}
\par
\label{sec:convergenceG}
In this section, we analyze the ideal per-batch gradient on fixed parameter and output neighborhoods. A stopping-time argument and a power-type gradient inequality give a quantitative probability of remaining in the parameter neighborhood for all iterations. On the corresponding event $E_G$, we obtain linear convergence for exponent $1/2$ and polynomial convergence for exponents in $(1/2,1)$. Local strong convexity in output space is an important exponent-$1/2$ special case; convexity gives the stopped-process auxiliary bound in Remark~\ref{r32}.

The central regularity condition is the gradient-form Kurdyka--\L{}ojasiewicz (KL) inequality on the fixed output domain $D$, stated in the standing assumptions below. The exponent $\alpha=1/2$ yields geometric convergence, whereas $\alpha\in(1/2,1)$ yields the polynomial exponent $(2\alpha-1)^{-1}$. Strong convexity of $F$ in output space with constant $c_Y>0$ implies this condition with $\alpha=1/2$ and $c=\sqrt{2c_Y}$, while mere convexity does not imply it.

To formalize this analysis, the section proceeds in two parts. Section~\ref{subsec:notation} establishes 
the required notation and standing assumptions; Section~\ref{subsec:theorems} then states and proves the 
convergence theorems under these assumptions.
\subsection{Notation and standing assumptions.}
\label{subsec:notation}
In this section, we adopt the same notation as in sections~\ref{subsec:setup} and~\ref{subsec:grad}
(in particular $Y\in\mathbb{R}^{N\times d}$, $\nabla_Y F$, per-sample Jacobians $\nabla_{\theta} Y_{n}$, and the definition $G = \nabla_{\mathcal{B}}F(Y) \times \nabla_{\theta} Y$).
For later reference, define $$z_{i}:=\left((\nabla_{Y}F)_{i}\cdot\nabla_{\theta}Y_{i}\right)^{\T},$$ so that the 
(transposed) batch gradient used by the algorithm is
$ G^{\T} = \sum\limits_{i\in\mathcal{B}} z_{i}$. If we use $\theta^{+}$ to represent the updated parameters,
with a fixed step size $\eta>0$, the parameter update is
$$ \theta^{+} = \theta - \eta\cdot G^{\T}=\theta - \eta\cdot\sum\limits_{i\in\mathcal{B}} z_{i},$$
and to simplify the subsequent proofs,  we assume all mini-batches have the same size $b=|\mathcal{B}|$. The results could be extended to variable batch sizes.

For $r>0$, write
\[
\begin{aligned}
B(\theta_0,r)&:=\{\theta\in\mathbb R^s:\|\theta-\theta_0\|_2<r\},\\
\overline B(\theta_0,r)&:=\{\theta\in\mathbb R^s:\|\theta-\theta_0\|_2\leq r\}
\end{aligned}
\]
for the open and closed Euclidean balls, respectively. Parameter norms without a subscript also denote the Euclidean norm.
Fix a deterministic initialization $\theta_0$ and $R>0$, and prescribe the parameter regions
\[
Q_{\mathrm{in}}=B(\theta_0,R),\qquad Q_{\mathrm{out}}=\overline B(\theta_0,2R).
\]
The map $Y$ is twice continuously differentiable on an open neighborhood of $Q_{\mathrm{out}}$. Let $D$ be a convex open output domain containing $Y(Q_{\mathrm{out}})$ and a point $Y^*$, with $F$ continuously differentiable on $D$. Write
\[
\begin{gathered}
F^*=F(Y^*),\quad f(Y)=F(Y)-F^*\geq0\quad(Y\in D),\\
f_k=f(Y(\theta_k)),\quad g_k=\nabla_YF(Y(\theta_k)).
\end{gathered}
\]
In particular, $\nabla_YF(Y^*)=0$. The geometry of $F$ is imposed in output space; the parameter regions independently specify where the Jacobian bounds hold.

Put $p=b/N\in(0,1]$ and $\mathcal F_k=\sigma(\theta_0,\mathcal B_0,\ldots,\mathcal B_{k-1})$. At each step, $\mathcal B_k$ is sampled uniformly without replacement from all size-$b$ subsets, independently of $\mathcal F_k$. Thus
\[
\E[G_k^{\T}\mid\mathcal F_k]=p\sum_{i=1}^N z_{i,k}.
\]
No independence assumption on the dataset is needed for this sampling identity. Batches may contain the same index at different iterations. We write $\E_{\mathcal B}[\cdot]=\E[\cdot\mid\mathcal F_k]$ in single-step statements. The following assumptions hold uniformly on $D$ for output-space quantities and on $Q_{\mathrm{out}}$ for parameter derivatives.

\textbf{Assumptions.}
\begin{itemize}
    \item($\text{A}_{1}$) $\|\nabla_Y F(Y)\|_{\F}\leq B_{0}$ for all $Y\in D$.
    \item($\text{A}_{2}$) The map $Y\mapsto\nabla_Y F(Y)$ is $L$-Lipschitz in Frobenius norm; i.e. for all $Y,Y'\in D$,
	$$\|\nabla_Y F(Y) -\nabla_Y F(Y')\|_{\F}\leq L\|Y - Y'\|_{\F}.$$ (Equivalently, $\nabla_Y F$ has Lipschitz constant $L$.)
    \item($\text{A}_{3}$) $\|\nabla_{\theta}Y(\theta)\|_{\F}\leq B_1$ for all $\theta\in Q_{\mathrm{out}}$.
    \item($\text{A}_{4}$) $\|\nabla^{2}_{\theta}Y(\theta)\|_{\F}\leq B_2$ for all $\theta\in Q_{\mathrm{out}}$.
    \item($\text{A}_{5}$) The smallest eigenvalue of $\nabla_{\theta}Y\cdot(\nabla_{\theta}Y)^{\T}$ admits a uniform positive lower bound on $Q_{\mathrm{out}}$:
    $$\lambda_{\min}\left( \nabla_{\theta}Y\cdot(\nabla_{\theta}Y)^{\T}\right)\geq\lambda_{\min}>0.$$
\end{itemize}
For the infinite-horizon confinement results we additionally impose
\begin{equation}\label{eq:KL}
\tag{KL}
\|\nabla_YF(Y)\|_{\F}\geq c f(Y)^\alpha\quad(Y\in D),\qquad c>0,\quad\alpha\in[1/2,1).
\end{equation}
The power-type gradient inequality is the gradient form of a Kurdyka--\L{}ojasiewicz (KL) condition. Strong convexity of $F$ on $D$ with constant $c_Y>0$ implies it with $\alpha=1/2$ and $c=\sqrt{2c_Y}$. Assumption $\text{A}_6$, specific to the surrogate, is stated in Section~\ref{subsec:notation_2}.

\begin{remark}\label{rmk:limit}
Assumption $\text{A}_5$ concerns the row-wise vectorized full-data Jacobian $J(\theta)\in\mathbb R^{Nd\times s}$. It requires full row rank, with all its $Nd$ singular values bounded below by $\sqrt{\lambda_{\min}}$. A bound only on the smallest nonzero singular value would permit rank deficiency and is insufficient. The bound is assumed on the prescribed region $Q_{\mathrm{out}}$, not merely along an unknown optimization path. This local structural condition is what transfers output-space gradient information into a nontrivial parameter-space descent direction within the region where the stopped estimates are applied. It is not asserted as a global property of every neural parameterization.

Let $K(\theta)=J(\theta)J(\theta)^\top$ be the unnormalized full-data NTK Gram matrix. If
\[
\lambda_{\min}(K(\theta_0))\geq\lambda_0>0,
\]
then, by Assumption A4 and the convexity of $Q_{\mathrm{out}}$,
\[
\|J(\theta)-J(\theta_0)\|_{\mathrm{op}}
\leq B_2\|\theta-\theta_0\|
\leq 2B_2R,\qquad \theta\in Q_{\mathrm{out}}.
\]
Hence, provided $2B_2R<\sqrt{\lambda_0}$,
\[
\lambda_{\min}(K(\theta))
\geq \bigl(\sqrt{\lambda_0}-2B_2R\bigr)^2
\quad\text{on }Q_{\mathrm{out}}.
\]
In particular, if $2B_2R\leq\sqrt{\lambda_0}/2$, then A5 holds on $Q_{\mathrm{out}}$ with $\lambda_{\min}=\lambda_0/4$. Thus an initial spectral margin together with local Jacobian stability provides a sufficient condition for A5 to hold uniformly on the prescribed region. Consequently, the resulting spectral lower bound applies throughout any trajectory that remains in $Q_{\mathrm{in}}\subset Q_{\mathrm{out}}$. NTK or overparameterization results may provide the initial spectral event under their respective assumptions~\cite{banerjee2023ntk,nguyen2021ntk}.
\end{remark}
\begin{remark}[Directional alternative to A5]\label{rmk:A5prime}
The full-row-rank condition in A5 is a convenient sufficient condition, but it is stronger than what the descent estimates use pointwise. A possible weaker alternative is the directional condition
\[
\left\|J(\theta)^{\top}\operatorname{vec}\!\left(\nabla_YF(Y(\theta))\right)\right\|^{2}
\geq \lambda_{\mathrm{dir}}
\left\|\nabla_YF(Y(\theta))\right\|_{\F}^{2},
\qquad \theta\in Q_{\mathrm{out}},
\]
for some $\lambda_{\mathrm{dir}}>0$. It requires the actual output-gradient direction to remain visible through the parameterization, whereas A5 controls every output direction uniformly on $Q_{\mathrm{out}}$. The main results are stated under A5.
\end{remark}
\begin{equation}\label{eq:G_margin}
\eta B_1B_0\leq R.
\end{equation}
Lemma~\ref{l31} then ensures that any update starting in $Q_{\mathrm{in}}$ has its entire parameter line segment in $Q_{\mathrm{out}}$, including the first exit update. Its output endpoints lie in $D$ and their chord lies in $D$. Lemmas~\ref{l32}--\ref{l33} are applied only to such steps. Define
\[
\tau_G=\inf\{k\geq0:\|\theta_k-\theta_0\|\geq R\},\quad
Z_k=f(Y(\theta_{k\wedge\tau_G})),\quad E_G=\{\tau_G=\infty\},
\]
with $\inf\varnothing=\infty$. For the analysis the process is frozen at exit; no local assumptions are invoked afterwards.
\subsection{Main theorems and lemmas.}
\label{subsec:theorems}
The results below separate single-step estimates, confinement, and KL rates. Lemmas~\ref{l31}--\ref{l33} give the single-step estimates and stopped stationarity bound. Under the power-type KL inequality \eqref{eq:KL}, Theorems~\ref{t34} and~\ref{t36} give high-probability confinement and convergence on $E_G$, with exponent-dependent rates. Mere convexity is treated in Remark~\ref{r32}.

The following lemma relates the per-step actual gradient estimator to the full gradient, providing a norm inequality that will be used repeatedly in the single-step bounds.
\begin{lemma}
\label{l31}
    If Assumption $\text{A}_{3}$ holds, then $\|G\|_{\F}\leq B_{1}\|\nabla_{Y}F\|_{\F}$. Moreover, for a uniformly sampled batch and conditional on $\mathcal F_k$,
    \begin{equation*}
    \E[\|G_k\|_{\F}^{2}\mid\mathcal F_k]\leq pB_1^2\|g_k\|_{\F}^{2},
    \qquad
    \E[\|G_k\|_{\F}\mid\mathcal F_k]\leq pB_1\|g_k\|_{\F}.
    \end{equation*}
\end{lemma}
Then, we explain the single-step descent situation through Lemma~\ref{l32} and Lemma~\ref{l33}.
\begin{lemma}
\label{l32}
    If Assumptions $\text{A}_{1}$--$\text{A}_{4}$ hold at a current state $\theta\in Q_{\mathrm{in}}$ and the margin condition \eqref{eq:G_margin} holds, denote $Y^{+}=Y(\theta^{+})$. Then
    there exists a constant $c_{2}\geq 0$
    such that
    $$ F(Y^{+})\leq F(Y)-\eta\cdot\sum\limits_{i=1}^{N}\langle z_{i}, G^{\T}\rangle + c_{2}\cdot\eta^{2}\|\nabla_{Y}F\|_{\F}^{2}.$$
    Specifically, one may take $c_{2}= \frac{1}{2}\left(LB_{1}^{2}+B_{0}B_{2}\right)c_{1}$, where $c_{1}= B_{1}^2$.
\end{lemma}
\begin{lemma}\label{l33}
    If Assumptions $\text{A}_{1}$--$\text{A}_{5}$ hold at a current state $\theta_k\in Q_{\mathrm{in}}$, the margin condition \eqref{eq:G_margin} holds, and $c_{3}:=p(\lambda_{\min}-c_{2}\eta)>0$,
    then the stopped one-step inequality holds:
    \begin{equation*}
    \E[Z_{k+1}\mid\mathcal F_k]
    \leq Z_k-c_{3}\eta\|g_k\|_{\F}^{2}\mathbf 1_{\{k<\tau_G\}}.
    \end{equation*}
    The variables $Z_k$ are nonnegative and bounded on the stopped process, because the margin condition includes the exit step and $Y(Q_{\mathrm{out}})$ is compactly contained in $D$.
\end{lemma}
\begin{remark}\label{r31}
Up to Lemma~\ref{l33}, the analysis uses neither convexity nor the KL inequality. The resulting stopped stationarity estimate is
\begin{equation*}
\sum_{k=0}^{T}\E\left[\|g_k\|_{\F}^{2}\mathbf 1_{\{k<\tau_G\}}\right]
\leq \frac{f_0}{c_3\eta},\qquad
\min_{0\leq k\leq T}\E\left[\|g_k\|_{\F}^{2}\mathbf 1_{\{k<\tau_G\}}\right]
\leq \frac{f_0}{c_3\eta(T+1)}.
\end{equation*}
Consequently the minimum expected stopped squared gradient is at most $f_0/(c_3\eta(T+1))$, and the corresponding expected stopped gradient norm is at most $\sqrt{f_0/(c_3\eta(T+1))}$. A detailed telescope and Cauchy--Schwarz argument is given in Appendix~\ref{d_r31}.
\end{remark}
With the preceding lemmas in place, we state the confinement and KL results. Put
\[
\Phi(t)=\frac{t^{1-\alpha}}{c(1-\alpha)},\qquad d_G=c_3\eta c^2.
\]
\begin{theorem}\label{t34}
    Assume $\text{A}_{1}$--$\text{A}_{5}$, the KL condition \eqref{eq:KL}, the margin condition \eqref{eq:G_margin}, and $c_3>0$. If
    \[
    \frac{pB_1\Phi(f_0)}{c_3R}\leq\zeta<1,
    \]
    then $\mathbb P(E_G)\geq1-\zeta$, and
    \[
    \E\sum_{k<\tau_G}\eta\|g_k\|_{\F}\leq\frac{\Phi(f_0)}{c_3},\qquad
    \E\sum_{k<\tau_G}\|\theta_{k+1}-\theta_k\|\leq\frac{pB_1\Phi(f_0)}{c_3}.
    \]
    If $\alpha=1/2$, assume additionally $d_G<1$ and set $\rho_G=1-d_G$. Then $m_k:=\E[Z_k\mathbf1_{\{k<\tau_G\}}]$ satisfies $m_k\leq f_0\rho_G^k$. In all cases
    \[
    \E[Z_k\mathbf1_{E_G}]\leq m_k,
    \qquad \E[Z_k\mid E_G]\leq\frac{m_k}{1-\zeta}.
    \]
    On $E_G$, the iterates converge to a finite-length limit $\theta_\infty\in\overline{Q_{\mathrm{in}}}\subset Q_{\mathrm{out}}$ and $f(Y(\theta_\infty))=0$.
\end{theorem}
\begin{corollary}\label{coro35}
    Under Theorem~\ref{t34} with $\alpha=1/2$, for every fixed $\widehat\rho\in(\rho_G,1)$, almost every $\omega\in E_G$ satisfies
    \[
    f_k(\omega)\leq\widehat\rho^k
    \quad\text{for all sufficiently large }k.
    \]
\end{corollary}
\begin{theorem}\label{t36}
    Assume the hypotheses of Theorem~\ref{t34}, with $\alpha\in(1/2,1)$, and put $\gamma=(2\alpha-1)^{-1}$. If $f_0>0$, then
    \[
    m_k\leq\left[f_0^{-(2\alpha-1)}+(2\alpha-1)d_Gk\right]^{-\gamma}.
    \]
    If $f_0=0$, the initial point is fixed and all gaps are zero. On $E_G$, the same estimate holds with conditional expectation divided by $1-\zeta$; no unconditional assertion about paths outside $E_G$ is made.
\end{theorem}
\begin{corollary}\label{coro37}
    Under Theorem~\ref{t36}, almost every $\omega\in E_G$ satisfies $f_k(\omega)\to0$. Moreover, there is an increasing random subsequence $k_n(\omega)$ and a finite random $C(\omega)$ such that

    \[
    f_{k_n(\omega)}(\omega)\leq C(\omega)k_n(\omega)^{-\gamma}.
    \]
    In addition, for every fixed $\epsilon\in(0,\gamma)$,
    \[
    f_k=O_\omega\!\left(k^{-\gamma+\epsilon}\right)
    \quad\text{almost surely on }E_G.
    \]
    These pathwise conclusions hold on $E_G$.
\end{corollary}
\begin{remark}\label{r32}
    This convex auxiliary statement uses the stopped recursion and does not invoke the KL condition \eqref{eq:KL}. If $F$ is convex on $D$ and $\|Y(\theta)-Y^*\|_{\F}\leq B$ on $Q_{\mathrm{out}}$, then
    \[
    f(Y(\theta))\leq B\|\nabla_YF(Y(\theta))\|_{\F}.
    \]
    Consequently the stopped recursion gives $m_{k+1}\leq m_k-(c_3\eta/B^2)m_k^2$ and hence $m_k\leq B^2/(c_3\eta k)$ for $k\geq1$. Under Theorem~\ref{t36}, for $\epsilon\in(0,\gamma)$, the deterministic indices
    \[
    k_n=\lfloor n^{2/\epsilon}\rfloor+n
    \]
    satisfy $f_{k_n}\leq C(\omega)k_n^{-\gamma+\epsilon}$ almost surely on $E_G$. The detailed derivation is in Appendix~\ref{d_r32}.
\end{remark}
\section{The local convergence for the cached surrogate gradient}
\par
We analyze the cached surrogate with a Lyapunov function coupling the objective gap and cache error. The resulting stopped descent yields local confinement, finite parameter length on $E_H$, and linear or polynomial rates determined by the KL exponent.
The section is organized in two parts: Section~\ref{subsec:notation_2} introduces additional notation and standing assumptions; Section~\ref{subsec:theorems_2} 
states and proves the corresponding convergence results.
\subsection{Notation and standing assumptions.}
\label{subsec:notation_2}
We adopt the notation and assumptions from Section~\ref{sec:convergenceG} (subsections~\ref{subsec:notation} and~\ref{subsec:grad}). Let $\theta_k$ denote the 
parameters at iteration $k$, and, for brevity, set 
$$\theta \coloneqq \theta_k,\quad \theta^{+} \coloneqq \theta_{k+1}, \quad \theta^{-} \coloneqq \theta_{k-1},$$
and correspondingly
$$Y \coloneqq Y(\theta),\quad Y^{+} \coloneqq Y(\theta^{+}),\quad Y^{-} \coloneqq Y(\theta^{-}).$$
Furthermore, we have $\tilde{Y}$ at step $k$ and $\tilde{Y}^{+}$ at step $k+1$ satisfying
\[ (\tilde{Y})_{n} = \left\{ 
     \begin{array}{ll}
        (Y)_{n} & \text{if }  n \in \mathcal{B}^{-} \\
       (\tilde{Y}^{-})_{n} & \text{if } n \notin \mathcal{B}^{-}
     \end{array}
\right. \]
and 
\[ (\tilde{Y}^{+})_{n} = \left\{ 
     \begin{array}{ll}
        (Y^{+})_{n} & \text{if }  n \in \mathcal{B} \\
       (\tilde{Y})_{n} & \text{if } n \notin \mathcal{B}
     \end{array},
\right. \]
where $\mathcal{B}^{-}$ is the mini-batch sampled at the step following $\theta^{-}$, and $\mathcal{B}$ is the 
mini-batch sampled at the step following $\theta$.

For notational convenience define 
$$p_{i}=\left((\widetilde{\nabla_{Y}F})_{i}\cdot\nabla_{\theta}Y_{i}\right)^{\T}\in\mathbb{R}^{s}, \quad \widetilde{\nabla_{Y}F} = \widetilde{\nabla_{Y}F}(\tilde{Y},Y).$$
Then $ H^{\T} = \sum\limits_{i\in\mathcal{B}} p_{i}$ and with fixed step size $\eta>0$ the parameter update is 
$$ \theta^{+} = \theta - \eta\cdot H^{\T}=\theta - \eta\cdot\sum\limits_{i\in\mathcal{B}} p_{i},$$
As in Section~\ref{sec:convergenceG}, all mini-batches have a common size $b=|\mathcal{B}|$, and $p=b/N\in(0,1]$. The initialization $\theta_0$ is deterministic and the cache is initialized exactly: $\tilde{Y}_0=Y(\theta_0)$. Let $\mathcal{F}_k$ be the history before sampling $\mathcal{B}_k$. Conditional on $\mathcal{F}_k$, this batch is a uniform size-$b$ subset, independent of the history. In particular, the full surrogate
\[
 h_k=h(\tilde{Y}_k,Y_k):=\widetilde{\nabla_Y F}(\tilde{Y}_k,Y_k)
\]
is determined before sampling the current batch. The selected cache rows are refreshed after the parameter update using $Y_{k+1}$.

We use the deterministic domains from Section~\ref{sec:convergenceG},
\[
 Q_{\mathrm{in}}=B(\theta_0,R),\qquad
 Q_{\mathrm{out}}=\overline{B}(\theta_0,2R),\qquad R>0.
\]
The convex open output domain $D$ contains $Y(Q_{\mathrm{out}})$ and $Y^*$. Write $F^*=F(Y^*)$, $f(Y)=F(Y)-F^*\geq0$ on $D$, $f_k=f(Y_k)$, $g_k=\nabla_Y F(Y_k)$, and $e_k=\tilde{Y}_k-Y_k$. Assumptions $\text{A}_1$--$\text{A}_2$ hold uniformly on $D$, and $\text{A}_3$--$\text{A}_5$ uniformly on $Q_{\mathrm{out}}$. Under row-wise vectorization, $J(\theta)$ is the $Nd\times s$ output Jacobian, and $J J^{\T}\succeq\lambda I$ with $\lambda=\lambda_{\min}>0$ as in Assumption $\text{A}_5$; thus full row rank is required. For the confinement and rate results we use the KL condition \eqref{eq:KL} from Section~\ref{sec:convergenceG}.

Treating $\widetilde{\nabla_Y F}(\tilde{Y},Y)$ as a function of $\tilde{Y}$ with $Y$ held fixed, we consider its first-order approximation. Using the integral form of Taylor's theorem, we have
\begin{equation*}
\begin{aligned}
\widetilde{\nabla_{Y}F}(\tilde{Y},Y)
&= \widetilde{\nabla_{Y}F}(Y,Y) + \int_{0}^{1}\mathcal{T}(Y+t(\tilde{Y}-Y))[\tilde{Y}-Y]\mathrm{d}t\\
&= \nabla_{Y}F + \int_{0}^{1}\mathcal{T}(Y+t(\tilde{Y}-Y))[\tilde{Y}-Y]\mathrm{d}t,
\end{aligned}
\end{equation*}
where for each $\Theta$, $\mathcal{T}(\Theta)$ is a fourth-order tensor and $\mathcal{T}(\Theta)[\Delta]$ denotes the contraction of this tensor with the matrix 
$\Delta=\tilde{Y}-Y$ (i.e. the linear operator induced by the tensor acting on matrices). 

Complementing Assumptions $\text{A}_{1}$-$\text{A}_{5}$ established from Section~\ref{sec:convergenceG}, 
we introduce the following additional standing assumption for the analysis in this section:
\begin{itemize}
    \item($\text{A}_{6}$) There exist $r>0$ and $B_3\geq0$ such that $h(Y,Y)=\nabla_Y F(Y)$ for every $Y=Y(\theta)$, $\theta\in Q_{\mathrm{out}}$. For each such $Y$ and every $\tilde{Y}$ with $\|\tilde{Y}-Y\|_{\F}\leq r$, $h(\cdot,Y)$ is defined and continuously differentiable on a neighborhood of the full segment $Y+t(\tilde{Y}-Y)$, $0\leq t\leq1$, with
    \[
    \|D_1h(Y+t(\tilde{Y}-Y),Y)\|_{\mathrm{op}}\leq B_3.
    \]
    The operator norm is induced by Frobenius norms. The surrogate is defined and continuously differentiable on these cache segments. In the preceding Taylor formula, $\mathcal{T}(\Theta)=D_1h(\Theta,Y)$ with the second argument $Y$ held fixed. In particular,
    \[
    \|h-g\|_{\F}\leq B_3\|e\|_{\F},\qquad
    \|h\|_{\F}\leq B_0+B_3r,\qquad
    \|h\|_{\F}^2\leq2\|g\|_{\F}^2+2B_3^2\|e\|_{\F}^2.
    \]
\end{itemize}
For clarity, fix the following constants independently of the iteration (the scalar $s_0$ is unrelated to the parameter dimension $s$):
\[
 \begin{aligned}
 s_0&=2/p,& q&=(1+1/s_0)(1-p),& d_0&=1-q=p(1+p)/2,\\
 C_H&=(1+s_0)B_1^4,& D_H&=\frac{pB_1^4B_3^2}{2\lambda},&
 \beta&=\frac{4(D_H+1)}{d_0},\\
 A_H&=c_2+\beta C_H,& a_H&=p\lambda/4,&\kappa&=d_0/2,
 \end{aligned}
\]
where $c_2=\frac12(LB_1^2+B_0B_2)B_1^2$ is the constant in Lemma~\ref{l32}. Throughout the following analysis the fixed step size satisfies
\begin{equation}\label{eq:H_stepsize}
 0<\eta\leq1,\qquad \eta\leq\frac{p\lambda}{8A_H},\qquad
 2A_HB_3^2\eta^2\leq\frac{\beta d_0}{4},\qquad
 \eta B_1(B_0+B_3r)\leq R.
\end{equation}
These conditions admit sufficiently small positive step sizes without division by $B_3$ and also cover $p=1$, for which $q=0$. They are sufficient conditions based on uniform bounds over the prescribed local region. Combining these bounds with Young's inequality can make the resulting step-size restrictions conservative relative to the behavior along a particular trajectory. Define the nonnegative Lyapunov function
\[
 V_k=f_k+\beta\|e_k\|_{\F}^2.
\]
If the current state satisfies $\theta_k\in Q_{\mathrm{in}}$ and $\|e_k\|_{\F}\leq r$, then Lemma~\ref{l41} and the last condition in \eqref{eq:H_stepsize} give $\|\theta_{k+1}-\theta_k\|\leq R$. Hence the entire parameter step segment lies in $Q_{\mathrm{out}}$, including a possible first exit from $Q_{\mathrm{in}}$. Both output endpoints belong to $D$, and their chord lies in $D$ by convexity of the domain.

\subsection{Main theorems and lemmas.}
\label{subsec:theorems_2}
The first three lemmas are analogues of Lemmas~\ref{l31}--\ref{l33} for the surrogate output-gradient and the batch quantity $H$. Their one-step assertions are used at current states $\theta_k\in Q_{\mathrm{in}}$, $\|e_k\|_{\F}\leq r$, under the margin in \eqref{eq:H_stepsize}; the preceding segment argument justifies their Taylor estimates even for the first exit step. Full details are given in Appendix~\ref{d_l41}--\ref{d_l46}. Lemmas~\ref{l44}--\ref{l46} establish Lyapunov descent, the stopped supermartingale, and the KL length bound. Theorem~\ref{t47} gives confinement for all $\alpha\in[1/2,1)$ and linear rates when $\alpha=1/2$; Theorem~\ref{t49} gives the polynomial rates when $\alpha>1/2$.
\begin{lemma}\label{l41}
    If Assumption $\text{A}_{3}$ holds at a current state with $\theta_k\in Q_{\mathrm{in}}$ and $\|e_k\|_{\F}\leq r$, then $\|H\|_{\F}\leq B_{1}\|\widetilde{\nabla_{Y}F}\|_{\F}$.
\end{lemma}
\begin{lemma}\label{l42}
    At a current state with $\theta_k\in Q_{\mathrm{in}}$ and $\|e_k\|_{\F}\leq r$, under Assumptions $\text{A}_{1}$--$\text{A}_{4}$ along the parameter segment and the margin in \eqref{eq:H_stepsize}, using $c_2$ from Lemma~\ref{l32},
    we similarly have
    $$ F(Y^{+})\leq F(Y)-\eta\cdot\sum\limits_{i}\langle z_{i}, H^{\T}\rangle + c_{2}\cdot\eta^{2}\|\widetilde{\nabla_{Y}F}\|_{\F}^{2},$$
    where $z_{i}=\left((\nabla_{Y}F)_{i}\cdot\nabla_{\theta}Y_{i}\right)^{\T}$ is as defined in Section~\ref{subsec:notation}.
\end{lemma}
\begin{lemma}\label{l43}
    At a current state with $\theta_k\in Q_{\mathrm{in}}$ and $\|e_k\|_{\F}\leq r$, if Assumptions $\text{A}_{1}$--$\text{A}_{4}$ hold along the update segment, then for any scalar $u>0$, conditional on the history before the current batch,
    \[
    \E\!\left[\|\tilde{Y}^{+}-Y^{+}\|_{\F}^{2}\mid\mathcal{F}_k\right]
    \leq (1+1/u)(1-p)\|\tilde{Y}-Y\|_{\F}^{2}
    +\eta^{2}(1+u)B_1^4\|\widetilde{\nabla_Y F}\|_{\F}^{2}.
    \]
    The auxiliary scalar is denoted by $u>0$ throughout the recursion. The first cache update is the base case, obtained from the initialization and cache-refresh definition; subsequent recursive applications use positive values of $u$.
\end{lemma}
\begin{lemma}\label{l44}
Under Assumptions $\text{A}_1$--$\text{A}_6$ and \eqref{eq:H_stepsize}, if $\theta_k\in Q_{\mathrm{in}}$ and $\|e_k\|_{\F}\leq r$, then the next update, including a possible first exit, satisfies
\[
 \E[V_{k+1}\mid\mathcal{F}_k]
 \leq V_k-a_H\eta\|g_k\|_{\F}^2-\kappa\beta\|e_k\|_{\F}^2.
\]
This local conditional descent uses Assumptions $\text{A}_1$--$\text{A}_6$ and the step-size conditions only.
\end{lemma}
\begin{lemma}\label{l45}
Under the assumptions of Lemma~\ref{l44}, choose $0\leq V_0<\bar v\leq\beta r^2$ and define
\[
 \tau_H=\inf\{k\geq0:\ \|\theta_k-\theta_0\|\geq R\ \text{or}\ V_k\geq\bar v\},
 \qquad \inf\varnothing=\infty.
\]
For a finite exit, define the stopped process by
\[
 U_k=\begin{cases}
 V_k, & k<\tau_H,\\
 V_{\tau_H}, & k\geq\tau_H.
 \end{cases}
\]
If $\tau_H=\infty$, set $U_k=V_k$ for every $k$. Thus, after the first exit, $U_k$ is frozen and no local assumption is applied to subsequent updates. Then $U_k$ is a nonnegative, uniformly bounded supermartingale and
\[
 \E[U_{k+1}\mid\mathcal{F}_k]
 \leq U_k-\mathbf{1}_{\{k<\tau_H\}}
 \bigl(a_H\eta\|g_k\|_{\F}^2+\kappa\beta\|e_k\|_{\F}^2\bigr).
\]
At a finite exit, $\theta_{\tau_H}\in Q_{\mathrm{out}}$ and
\[
 \|e_{\tau_H}\|_{\F}\leq r+B_1^2\eta(B_0+B_3r).
\]
In particular, $\E U_k\leq V_0$ and $\Pr(\sup_{k\geq0}U_k\geq\bar v)\leq V_0/\bar v$.
\end{lemma}
\begin{lemma}\label{l46}
In addition to Lemma~\ref{l45}, assume the KL condition \eqref{eq:KL} with exponent $\alpha\in[1/2,1)$. Set
\[
 \begin{aligned}
 W_k&=\sqrt{\|g_k\|_{\F}^2+\|e_k\|_{\F}^2},&
 K_\alpha&=c^{-1}+\beta^\alpha r^{2\alpha-1},\\
 \Psi(t)&=\frac{K_\alpha}{1-\alpha}t^{1-\alpha},&
 m_H&=\min\{a_H,\kappa\beta\}>0.
 \end{aligned}
\]
Before exit, $V_k^\alpha\leq K_\alpha W_k$, and
\[
 \E\sum_{k<\tau_H}\eta W_k\leq\frac{\Psi(V_0)}{m_H},\qquad
 \E\sum_{k<\tau_H}\|\theta_{k+1}-\theta_k\|
 \leq\frac{B_1\sqrt{1+B_3^2}\,\Psi(V_0)}{m_H}.
\]
Consequently,
\[
 \Pr(\tau_H<\infty)
 \leq\frac{B_1\sqrt{1+B_3^2}\,\Psi(V_0)}{m_H R}
       +\frac{V_0}{\bar v}.
\]
The sums include the step that reaches a finite exit. At $V_k=0$ one has $g_k=e_k=h_k=0$, and no derivative of $\Psi$ at zero is used.
\end{lemma}
\begin{remark}
The stopped descent in Lemma~\ref{l45} also gives a stationarity estimate under the assumptions preceding the KL condition. Summing through time $T$ and using nonnegativity of $U_{T+1}$ yields
\[
 \sum_{k=0}^T\E\!\left[\mathbf{1}_{\{k<\tau_H\}}\|g_k\|_{\F}^2\right]
 \leq\frac{V_0}{a_H\eta}.
\]
Thus the minimum of these stopped second moments is at most $V_0/(a_H\eta(T+1))$, and Cauchy--Schwarz gives
\[
 \min_{0\leq k\leq T}\E\!\left[\mathbf{1}_{\{k<\tau_H\}}\|g_k\|_{\F}\right]
 \leq\sqrt{\frac{V_0}{a_H\eta(T+1)}}.
\]
These are stopped estimates, not unconditional stationarity guarantees for iterates after exit.
\end{remark}
\begin{theorem}\label{t47}
Assume $\text{A}_1$--$\text{A}_6$, the KL condition \eqref{eq:KL}, exact deterministic cache initialization, and \eqref{eq:H_stepsize}. Choose $0\leq V_0<\bar v\leq\beta r^2$ so that, for some $0\leq\zeta<1$,
\[
 \frac{B_1\sqrt{1+B_3^2}\,\Psi(V_0)}{m_H R}+\frac{V_0}{\bar v}\leq\zeta.
\]
Then the confinement event $E_H=\{\tau_H=\infty\}$ satisfies $\Pr(E_H)\geq1-\zeta$. Define the killed expectation and its descent constant by
\[
 n_k=\E[U_k\mathbf{1}_{\{k<\tau_H\}}],\qquad
 d_H=\frac{m_H\eta}{K_\alpha^2}>0.
\]
For every $\alpha\in[1/2,1)$,
\begin{equation}\label{eq:H_killed_rate}
 n_{k+1}\leq n_k-d_H n_k^{2\alpha},\qquad n_0=V_0,
 \qquad \E[V_k\mathbf{1}_{E_H}]\leq n_k,
 \qquad \E[V_k\mid E_H]\leq\frac{n_k}{1-\zeta}.
\end{equation}
Products with $\mathbf{1}_{E_H}$ can equivalently be written using $U_k$, with the stopped process supplying the values after exit. If $\alpha=1/2$, impose also $d_H<1$ and put $\rho_H=1-d_H\in(0,1)$. Then
\[
 n_k\leq V_0\rho_H^k,\qquad
 \E[V_k\mathbf{1}_{E_H}]\leq V_0\rho_H^k,\qquad
 \E[V_k\mid E_H]\leq\frac{V_0\rho_H^k}{1-\zeta}.
\]
For almost every path in $E_H$, for every allowed exponent, $V_k\to0$, $e_k\to0$, and the parameter path has finite total length. Consequently $\theta_k\to\theta_\infty\in\overline{Q_{\mathrm{in}}}$ with $f(Y(\theta_\infty))=0$.
\end{theorem}
The probability bound, the killed Jensen recurrence, and the finite-length argument follow from Lemmas~\ref{l44}--\ref{l46}; the full proof is given in Appendix~\ref{d_t47}.
\begin{corollary}\label{coro48}
Under Theorem~\ref{t47} with $\alpha=1/2$ and $d_H<1$, fix any deterministic $\widehat\rho\in(\rho_H,1)$. For almost every path $\omega\in E_H$, there exists $k_0(\omega)$ such that
\[
 V_k(\omega)\leq\widehat\rho^k\qquad(k\geq k_0(\omega)).
\]
In particular, the objective gap has this eventual linear bound and $\|e_k\|_{\F}\leq\beta^{-1/2}\widehat\rho^{k/2}$ eventually on $E_H$.
\end{corollary}
Markov's inequality and the Borel--Cantelli lemma applied to the linear killed expectation bound prove the assertion; details are given in Appendix~\ref{d_coro48}.
\begin{theorem}\label{t49}
Under the assumptions of Theorem~\ref{t47}, let $\alpha\in(1/2,1)$ and $\gamma=1/(2\alpha-1)$. If $V_0>0$, then
\[
 n_k\leq\left[V_0^{-(2\alpha-1)}+(2\alpha-1)d_H k\right]^{-1/(2\alpha-1)}.
\]
The same upper bound holds for $\E[V_k\mathbf{1}_{E_H}]$, and divided by $1-\zeta$ it bounds $\E[V_k\mid E_H]$. In particular, these expectations and the corresponding objective-gap expectations are $O(k^{-\gamma})$. If $V_0=0$, the iterates and exact cache remain fixed, $\tau_H=\infty$, and all these quantities are zero.
The polynomial bound follows by applying the convexity of $x\mapsto x^{-(2\alpha-1)}$ to the killed recurrence, with the zero-energy case handled separately; details are given in Appendix~\ref{d_t49}.
\end{theorem}
\begin{corollary}\label{coro410}
Under the assumptions of Theorem~\ref{t49}, for almost every path $\omega\in E_H$ the following hold:
\begin{itemize}
 \item $V_k\to0$, and hence $F(Y_k)-F^*\to0$ and $e_k\to0$; the parameter sequence converges with finite total length as in Theorem~\ref{t47}.
 \item There exist a strictly increasing random subsequence $k_n(\omega)\to\infty$ and a finite constant $C(\omega)>0$ such that
 \[
 V_{k_n}(\omega)\leq C(\omega) k_n^{-\gamma},\qquad
 \gamma=\frac{1}{2\alpha-1},\qquad n\in\mathbb{N}.
 \]
 Thus the same subsequence rate holds for the objective gap.
 \item For every fixed $\epsilon\in(0,\gamma)$,
 \[
 V_k=O_\omega\!\left(k^{-\gamma+\epsilon}\right)
 \quad\text{almost surely on }E_H.
 \]
 Thus, on $E_H$, $F(Y_k)-F^*=O_\omega(k^{-\gamma+\epsilon})$ and
 $\|e_k\|_{\F}=O_\omega(k^{-(\gamma-\epsilon)/2})$.
\end{itemize}
\end{corollary}
The assertion follows from Theorems~\ref{t47}--\ref{t49}, Fatou's lemma, and the resulting liminf argument; details are given in Appendix~\ref{d_coro410}.
\begin{remark}
A convex auxiliary estimate is also available. If $F$ is convex on $D$ and $\sup_{\theta\in Q_{\mathrm{out}}}\|Y(\theta)-Y^*\|_{\F}\leq B$, then $f\leq B\|g\|_{\F}$. Before exit,
\[
 V\leq B\|g\|_{\F}+\beta r\|e\|_{\F}\leq(B+\beta r)W.
\]
Lemma~\ref{l45} and the same killed Jensen argument give
\[
 n_{k+1}\leq n_k-\frac{m_H\eta}{(B+\beta r)^2}n_k^2,
\]
which yields $n_k=O(k^{-1})$ for the killed expectation. Confinement and pathwise rates for the cached iterates require the additional hypotheses of Theorem~\ref{t47} and its subsequent rate results.
\end{remark}

\section{Numerical experiments}
\par
\label{sec:experiments}
We use controlled numerical examples to study three aspects of the output-gradient construction: the effect of the output-loss geometry in a separable baseline, the cache-tracking behavior and update-loop cost for a Gram objective, and the effect of batch restriction for a pairwise objective with balanced external forces. The latter two experiments compare the ideal gradient $G$, the cached surrogate $H$, and two scalings of the batch-local gradient. All reported objective values are full-data gaps $F-F^*$, with $F^*$ computed analytically.

\subsection{Experimental protocol.}
All networks use the smooth activation
\[
\sigma(t)=0.01t+0.99\tau\log(1+\exp(t/\tau)),\qquad \tau=0.02,
\]
and double-precision arithmetic. This activation is smooth, including near the transition between its two approximately linear regimes. Each configuration is run with seeds $0,\ldots,4$. In Experiments 1 and 3 the seeds change both the data and the network initialization; Experiment 2 holds the data and target fixed and varies the initial output perturbation. Within each repeat of a coupled experiment, all four methods start from the same parameters and use the same sequence of random permutations.

A data epoch is one complete pass through a newly and independently permuted dataset, partitioned into consecutive batches. The last batch is retained when smaller than the nominal batch size. This random-reshuffling protocol visits every sample once per data epoch; the conditional sampling identities in Sections~\ref{sec:convergenceG} and~\ref{subsec:notation_2} concern independent uniform batches at individual steps. Experiments 1 and 3 use data epochs on the horizontal axis. In Experiment 2, one plotted epoch comprises ten data epochs, so the budget of $15{,}000$ data epochs is displayed as $1{,}500$ epochs. Its fitted exponential slopes use this same plotted unit.

For a batch $\mathcal B$, let $J_i(\theta)$ denote the Jacobian of the sample output and let $g(Y)=\nabla_YF(Y)$. The two full-output constructions use
\begin{equation}
\label{eq:exp-updates}
G_{\mathcal B}^{\T}=\sum_{i\in\mathcal B}J_i(\theta)^{\T}g_i(Y(\theta)),
\qquad
H_{\mathcal B}^{\T}=\sum_{i\in\mathcal B}J_i(\theta)^{\T}g_i(\widetilde Y).
\end{equation}
Both take a parameter step with learning rate $\eta$, without an additional $N/|\mathcal B|$ factor. The reported learning rates are fixed empirical choices, whereas \eqref{eq:H_stepsize} gives sufficient conditions in terms of uniform local constants. The cache is initialized by $\widetilde Y=Y(\theta_0)$. Each $H$ step evaluates the gradient using the old cache, updates the parameters, recomputes the current batch outputs at the new parameters, and replaces only those cache rows. The other rows remain unchanged until their next refresh.

The method \textit{naive-raw} differentiates the restriction of the full objective to the current batch, retaining the full-objective normalization constants. The method \textit{naive-q} uses precisely the same local objective with learning rate $\eta/q$, where
\begin{equation}
\label{eq:exp-q}
q=\frac{\sum_{\ell}b_\ell(b_\ell-1)}{N(N-1)}
\end{equation}
is the probability that a fixed pair of distinct samples belongs to the same batch during a random permutation epoch, and $b_\ell$ are the actual batch sizes. Thus the two naive methods differ only by a scalar step-size multiplier. This epoch-level pair probability is distinct from the single-step sample inclusion probability $b/N$.

We report means over the five runs in the coupled-objective figures and sample standard deviations in Table~\ref{tab:exp-results}. Frozen-gradient diagnostics hold the parameters fixed at each repeat's initialization and sum the batch parameter gradients over one independently generated permutation. For such a sum $\widehat v$ and the full parameter gradient $v$, we record $\|\widehat v-v\|/\|v\|$ and their cosine. In these diagnostics the naive-q sum is the naive-raw sum divided by $q$, matching the relative scaling of their parameter updates.

\subsection{Experiment 1: Separable losses with two KL exponents.}
The first experiment uses $N=10$ independent inputs from $[0,1]^4$ and independent targets from $[0,1]^2$, with a shared MLP of dimensions $4\to64\to64\to2$. Writing $u=\operatorname{vec}(Y-Y_{\mathrm{target}})\in\mathbb R^{20}$, the two objectives are
\[
F_2(Y)=\frac12\|u\|_2^2,\qquad
F_4(Y)=\frac14\|u\|_4^4,\qquad F_2^*=F_4^*=0.
\]
Their gradients satisfy
\[
\|\nabla F_2\|_2=\sqrt{2}\,F_2^{1/2},\qquad
\|\nabla F_4\|_2=\|u\|_6^3
\geq 20^{-1/4}(4F_4)^{3/4}.
\]
These give output-space KL exponents $1/2$ and $3/4$, respectively. With batch size one, the ideal update reduces to ordinary sample-wise SGD. The experiment therefore isolates the loss geometry before introducing sample coupling. We use $\eta=0.01$ for $10^5$ epochs for $F_2$, and $\eta=0.001$ for $5\times10^5$ epochs for $F_4$, recording endpoints every $100$ and $1{,}000$ epochs, respectively.

\begin{figure}[htbp]
\centering
\subcaptionbox{Quadratic loss.\label{fig:exp-mse}}
{\includegraphics[width=0.9\linewidth]{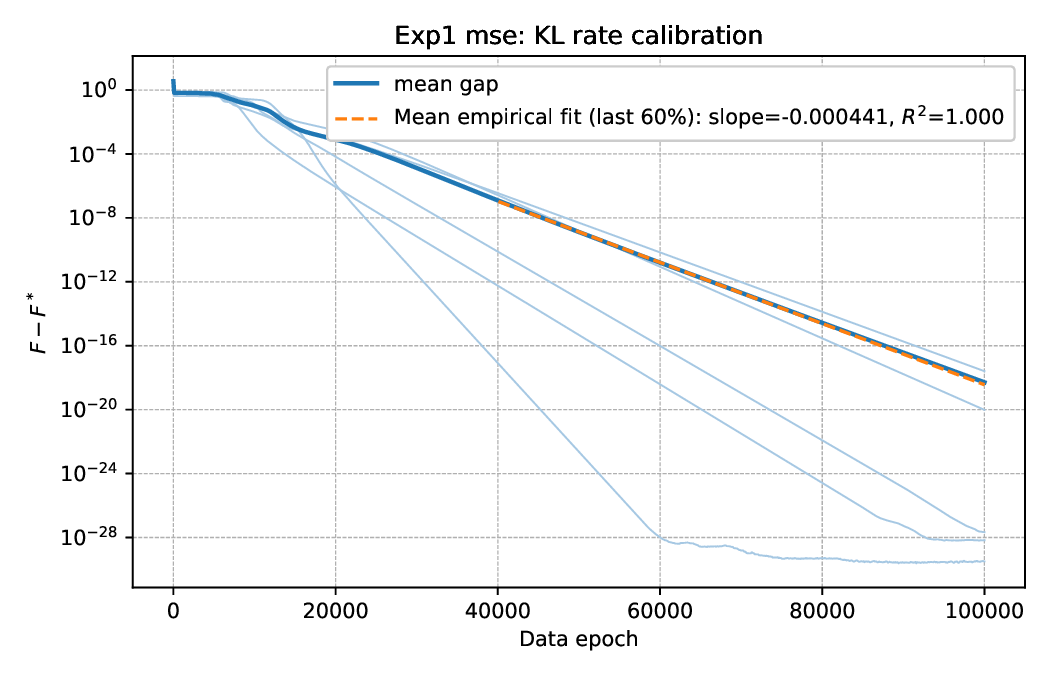}}
\par\medskip
\subcaptionbox{Fourth-power loss.\label{fig:exp-highpower}}
{\includegraphics[width=0.9\linewidth]{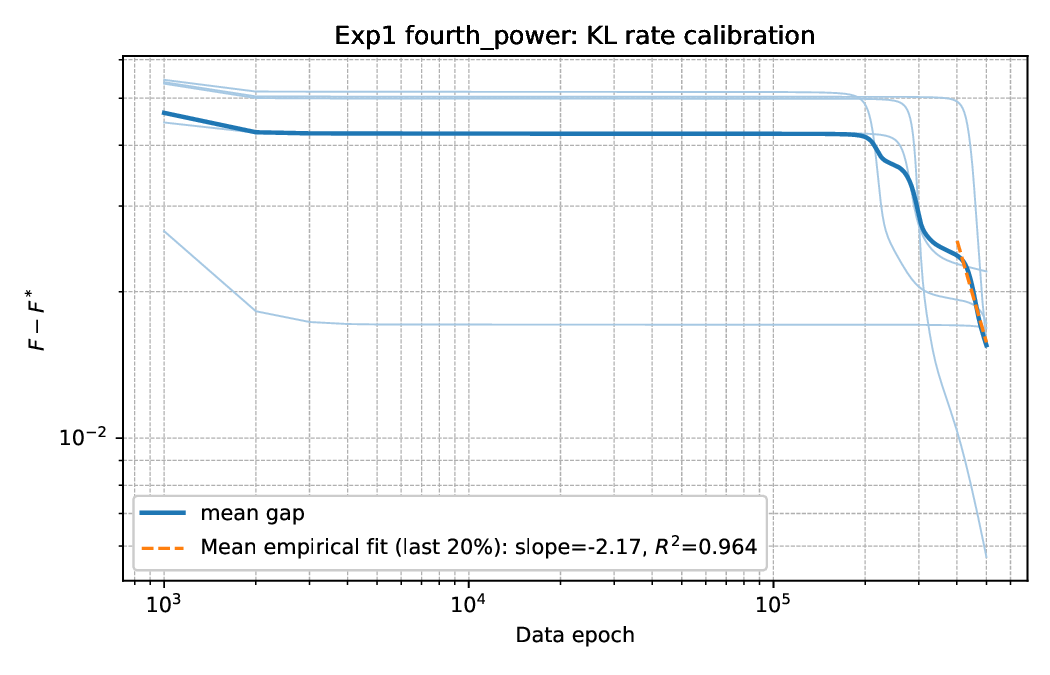}}
\caption{Separable-loss trajectories. Light blue curves show individual runs, the dark blue curve is their arithmetic mean, and the orange dashed line is an empirical fit to that mean. The quadratic panel uses a logarithmic vertical axis and the fourth-power panel uses logarithmic axes. Fits use the last $60\%$ and $20\%$ of stored records, respectively, with finite positive epochs and gaps above $10^{-14}$. The fitted mean slopes are approximately $-4.41\times10^{-4}$ in semilog coordinates and $-2.17$ in log--log coordinates, with $R^2$ approximately $1.000$ and $0.964$. The full trajectories are shown in both panels.}
\label{fig:exp-convergence}
\end{figure}

The quadratic trajectories exhibit a sustained exponential decay before reaching very small numerical values. Their final gaps range from $3.17\times10^{-30}$ to $2.55\times10^{-18}$. Four runs retain enough points above $10^{-14}$ in the selected tail window for individual fits, giving slopes between $-7.09\times10^{-4}$ and $-4.28\times10^{-4}$; the remaining run has already fallen below that threshold throughout the window. The mean fit in Fig.~\ref{fig:exp-mse} corresponds to $F_2\approx C\exp(-4.41\times10^{-4}k)$, with $k$ measured in data epochs.

For the fourth-power loss, the trajectories contain longer slowly changing intervals and transitions occurring at different epochs. The final gaps lie between $5.70\times10^{-3}$ and $2.20\times10^{-2}$, with mean $1.55\times10^{-2}$ and sample standard deviation $5.99\times10^{-3}$. The displayed $20\%$ window was chosen during postprocessing to describe the final part of these trajectories; the $30\%$, $40\%$, and $60\%$ window fits are also retained in the numerical records. Individual trajectories enter their sustained decay phases at different epochs, so the common fitting window can include plateau or transition segments for some runs. The mean-curve slope of $-2.17$ is therefore reported as a finite-window summary, alongside the $k^{-2}$ reference associated with exponent $3/4$.

\FloatBarrier
\subsection{Experiment 2: Rank-eight Gram approximation.}
We consider
\begin{equation}
\label{eq:exp-gram-objective}
F(Y)=\frac12\left\|\frac{YY^{\T}}{d}-S\right\|_{\F}^2,
\qquad N=200,\quad d=8.
\end{equation}
The inputs are standard Gaussian vectors in $\mathbb R^{32}$, and the output map is $16$ times a $32\to256\to8$ MLP. To construct the target, we evaluate a network with the same architecture after multiplying its last-layer weights and bias by $0.25$, obtaining $U\in\mathbb R^{200\times8}$. Let $z$ be a normalized Gaussian vector projected onto the orthogonal complement of the columns of $U$. We set
\[
S=\frac{UU^{\T}}8+\mu zz^{\T},\qquad
\mu=\frac12\lambda_{\min}^{+}\!\left(\frac{UU^{\T}}8\right),
\]
where $\lambda_{\min}^{+}$ is the smallest positive eigenvalue. The target has rank nine, while the model Gram matrix has rank at most eight. Spectral truncation therefore gives
\[
F^*=\frac12\|S-S_{+,8}\|_{\F}^2=\frac12\mu^2
=4.226244430859031,
\]
where $S_{+,8}=UU^{\T}/8$ and $\mu=2.907316436461305$. The analytic value agrees with the independently computed spectral residual to numerical precision.

Let $Y^*$ be a spectral factor of $S_{+,8}$. Each repeat starts from the prescribed output $Y^*+0.01\|Y^*\|_{\F}E$, where $E$ is a Gaussian matrix normalized to unit Frobenius norm. The target network's hidden layer is retained, and its last layer is fitted by least squares to realize this output. The initial full-data feature design, including the bias column, has rank $200$ and smallest singular value $0.3054$; the relative output fitting errors are below $6.1\times10^{-15}$. During training, the minimum recorded feature singular value over all runs and methods is $0.30508$, with rank $200$ throughout. Since the last-layer Jacobian is a scaled block copy of this full-data design and its Gram matrix is a positive-semidefinite contribution to $JJ^{\T}$, these measurements provide a trajectory-level diagnostic of the spectral lower bound relevant to A5 at the recorded parameter states.

We use batches of size $16$, with a final batch of size $8$, and
\[
\eta=\frac{10^{-2}}{16^4}=1.52587890625\times10^{-7},\qquad
q=0.0737688442.
\]
All methods run for $15{,}000$ data epochs. For the symmetric target, the exact output gradient is
\[
g(Y)=\frac{2}{8}\left(\frac{YY^{\T}}8-S\right)Y.
\]
The ideal method recomputes this quantity before each batch update. The cached method stores $\widetilde Y$ and $\widetilde K=\widetilde Y\widetilde Y^{\T}/8$, and uses
\[
H_{\mathcal B}^{\T}
=\sum_{i\in\mathcal B}J_i(\theta)^{\T}
\left[\frac{2}{8}(\widetilde K-S)\widetilde Y\right]_i.
\]
After refreshing $\widetilde Y_{\mathcal B}$ at the new parameters, it replaces the rows of $\widetilde K$ by $\widetilde Y_{\mathcal B}\widetilde Y^{\T}/8$ and the corresponding columns by their transpose. The naive objective is $\frac12\|Y_{\mathcal B}Y_{\mathcal B}^{\T}/8-S_{\mathcal B\mathcal B}\|_{\F}^2$.

The $G$ and $H$ trajectories in Fig.~\ref{fig:exp-gram} achieve similar reductions in both the gap and the output-gradient norm. Their mean final gaps are $5.57\times10^{-3}$ and $5.62\times10^{-3}$, respectively, compared with $1.95\times10^{-1}$ and $1.86\times10^{-1}$ for the two naive methods. The rank-eight optimum has a strict retained/discarded spectral gap; the resulting local quadratic growth transverse to the orthogonal factor symmetry motivates the exponential diagnostic in Fig.~\ref{fig:gram-GH}. Fitting $\log(F-F^*)$ against the plotted epoch over the final $60\%$ of records gives mean-curve slopes $-1.29\times10^{-3}$ for $G$ and $-1.28\times10^{-3}$ for $H$, both with $R^2\approx0.995$. Individual $G$ slopes range from $-1.38\times10^{-3}$ to $-1.19\times10^{-3}$, with a similar range for $H$.

\begin{figure}[htbp]
\centering
\includegraphics[width=\linewidth]{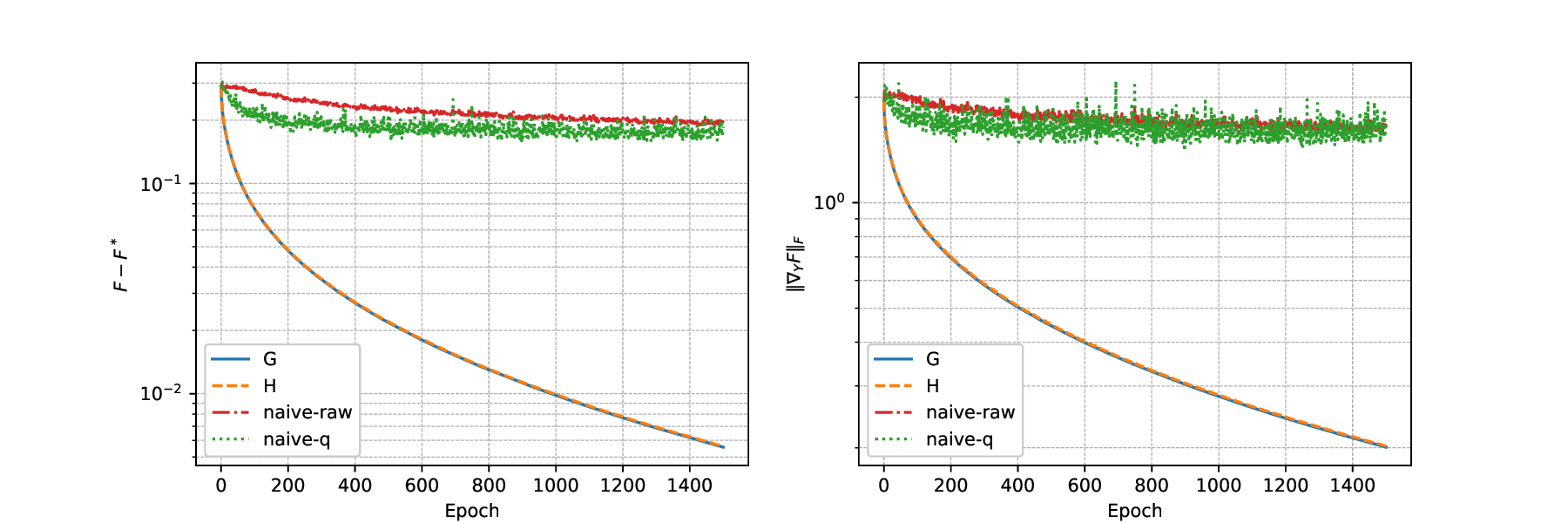}
\caption{Gram objective: mean full-data gap (left) and output-gradient norm (right), both on logarithmic vertical axes. Blue solid, orange dashed, red dash-dotted, and green dotted curves denote $G$, $H$, naive-raw, and naive-q, respectively. One plotted epoch is ten data epochs. Table~\ref{tab:exp-results} reports the variation in final gaps across the five runs.}
\label{fig:exp-gram}
\end{figure}

\begin{figure}[htbp]
\centering
\subcaptionbox{Output-cache discrepancy.}
{\includegraphics[width=0.9\linewidth]{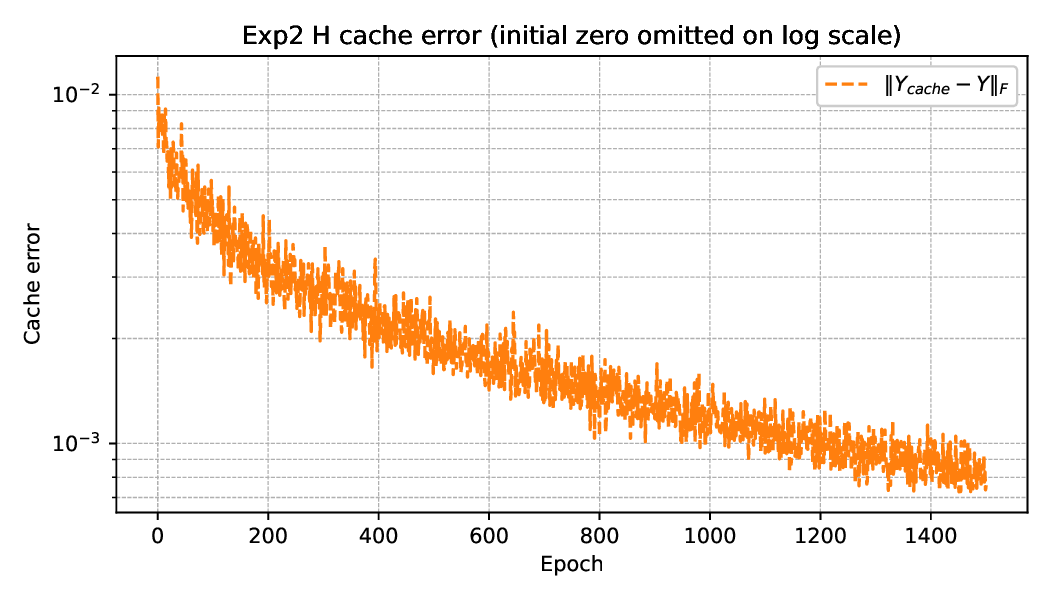}}
\par\medskip
\subcaptionbox{Exponential fits of the mean gap.}
{\includegraphics[width=0.9\linewidth]{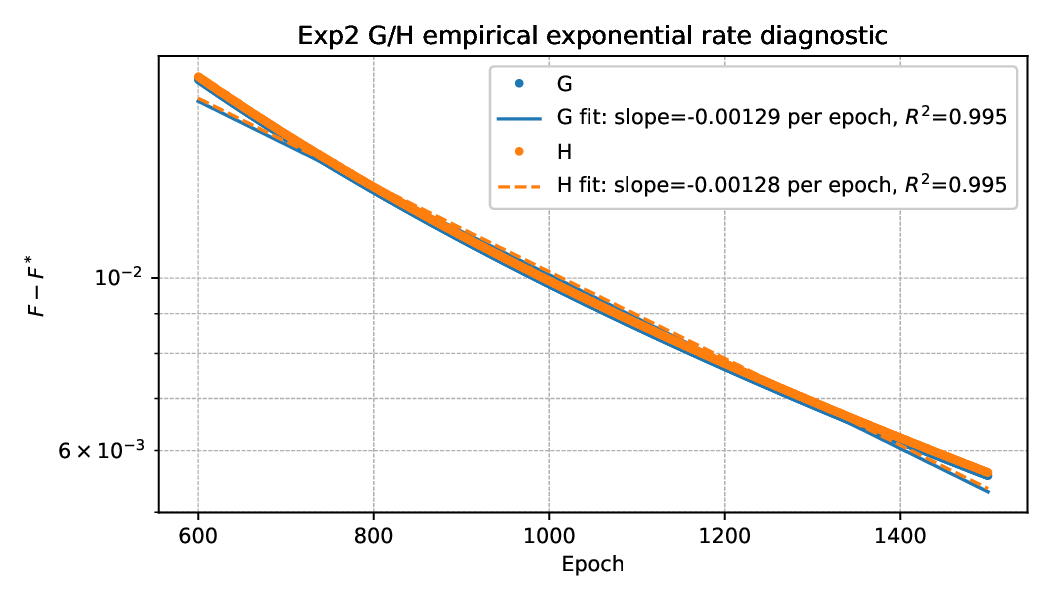}}
\caption{Gram cache and rate diagnostics. The cache curve is the mean of $\|\widetilde Y-Y\|_{\F}$ across runs. Its exact initial value is zero and is omitted on the logarithmic axis. The rate panel shows the retained observations and fitted curves over plotted epochs $600$--$1500$; slopes are measured per plotted epoch.}
\label{fig:gram-GH}
\end{figure}

The final output-cache errors range from $6.45\times10^{-4}$ to $8.35\times10^{-4}$. In the frozen diagnostics of Fig.~\ref{fig:exp-frozen}, the sum of ideal batch gradients recovers the full parameter gradient to relative error below $5.4\times10^{-16}$. The naive-raw relative errors range from $0.91$ to $1.06$, and the naive-q errors from $3.51$ to $5.41$. Their common direction has cosine between $-0.053$ and $0.415$ with the full gradient. Thus scalar pair-frequency compensation does not recover the full-gradient direction for this Gram construction.

\FloatBarrier
\subsection{Experiment 3: Pairwise coupling with balanced external forces.}
The third objective combines pair interactions and a sample-wise external force. Let $N=200$, $M=N(N-1)/2$, and let $\xi_i=1$ for half the samples and $\xi_i=-1$ for the other half. Define
\[
\phi(r)=\frac12r^2+\frac{\beta}{4}r^4,\qquad
h=\frac{N}{2M}\phi'(2A),
\]
and consider scalar outputs with objective
\begin{equation}
\label{eq:exp-pairwise-objective}
\begin{split}
F(y)={}&\frac1M\sum_{i<j}\phi\bigl(y_i-A\xi_i-y_j+A\xi_j\bigr)
+h\sum_i\xi_i y_i\\
&+\frac{\lambda}{4M}\sum_{i<j}\left(\frac{y_i+y_j}{2}\right)^4,
\qquad A=0.03,\quad\beta=\lambda=1.
\end{split}
\end{equation}
At $y=0$, the pair-force and external-force gradients cancel. The pairwise nonlinear terms and the quartic pair-sum terms are convex, while the external-force term is linear; hence the full objective is convex and this stationary point is a global minimizer, with
\[
F^*=\frac{(N/2)^2}{M}\phi(2A)
=9.061507537688442\times10^{-4}.
\]
The difference terms have positive quadratic curvature on $\mathbf1^\perp$, while along a common translation,
\[
F(t\mathbf1)-F^*=\frac\lambda4t^4,\qquad
\|\nabla F(t\mathbf1)\|_2=\frac\lambda{\sqrt N}|t|^3.
\]
To verify the local KL exponent, write $y=t\mathbf1+v$ with $\mathbf1^{\T}v=0$, and set $f(y)=F(y)-F^*$. Since $\phi''\geq1$ and the linear terms cancel at zero, the pair-difference and external-force terms together are bounded below, relative to their value at zero, by $\frac1{2M}\sum_{i<j}(v_i-v_j)^2$. Moreover, the average of $(v_i+v_j)/2$ over all pairs is zero, so Jensen's inequality for the quartic pair-sum term gives
\[
f(y)\geq\frac{N}{2M}\|v\|_2^2+\frac\lambda4t^4.
\]
For $0<f(y)\leq1$, it follows that
\[
\|y\|_2\leq\|v\|_2+\sqrt N|t|
\leq C f(y)^{1/4},\qquad
C=\sqrt{\frac{2M}{N}}+\sqrt N\left(\frac4\lambda\right)^{1/4}.
\]
Convexity and minimality at zero imply $f(y)\leq\langle\nabla F(y),y\rangle$, hence
\[
\|\nabla F(y)\|_2\geq C^{-1}f(y)^{3/4}.
\]
The inequality also holds at $y=0$. The common-translation identities above show sharpness: for any $\alpha<3/4$, the ratio $\|\nabla F(t\mathbf1)\|_2/f(t\mathbf1)^\alpha$ is a positive constant times $|t|^{3-4\alpha}$ and tends to zero as $t\to0$. Thus $3/4$ is the smallest admissible local KL exponent. Unlike the separable fourth-power example, the gradient here requires interactions with samples outside the current batch.

The shared MLP has dimensions $32\to256\to1$. The inputs are Gaussian except for their first coordinate, which is set to $\xi_i$. We initialize the output weights to zero and the output bias to one, so every run begins at $y=\mathbf1$. We use $\eta=0.025$, $6{,}000$ data epochs, and batches of size $12$, with a final batch of size $8$. These sizes give $q=0.0544723618$ and a naive-q learning rate of approximately $0.45894834$.

For completeness, the full output gradient is
\[
g_i(y)=h\xi_i+\frac1M\sum_{j\ne i}\phi'(z_i-z_j)
+\frac{\lambda}{16M}\sum_{j\ne i}(y_i+y_j)^3,
\qquad z_i=y_i-A\xi_i.
\]
The implementation evaluates the sums from first, second, and third moments, requiring linear work in $N$ for the full output gradient. The analytic formula agrees with automatic differentiation to maximum absolute error $1.52\times10^{-18}$ in the recorded check. The cached method evaluates this same formula at $\widetilde y$ before the current-batch parameter update. For each naive method, both pair sums are restricted to $i,j\in\mathcal B$ and the external sum to $i\in\mathcal B$, retaining $M$ and $h$.

\begin{figure}[htbp]
\centering
\subcaptionbox{Full-range mean objective gaps.}
{\includegraphics[width=0.9\linewidth]{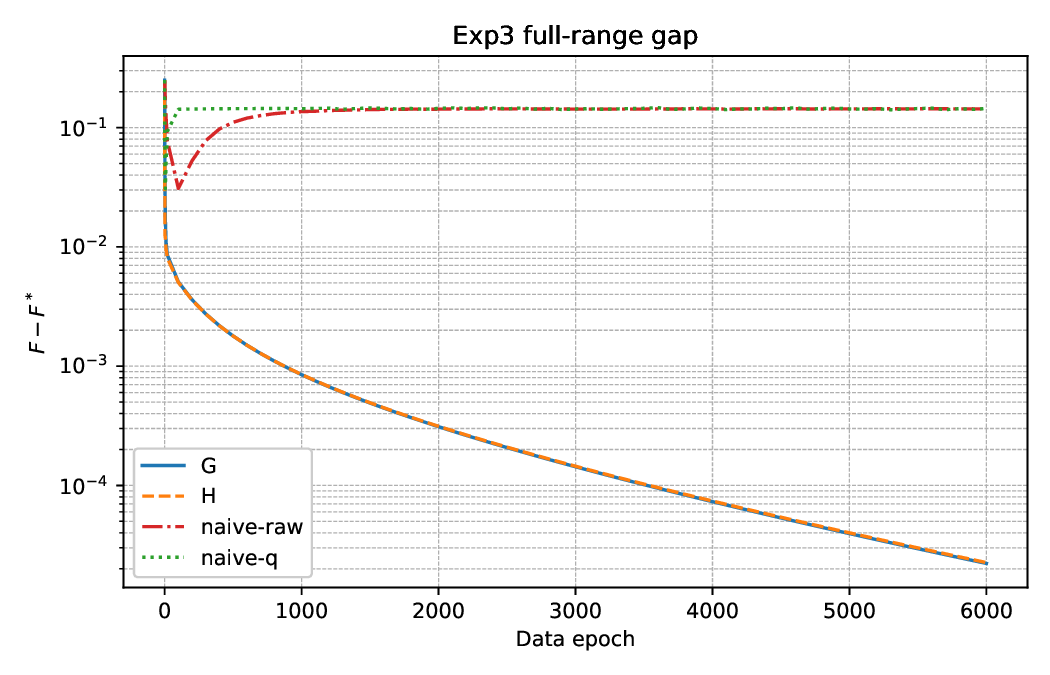}}
\par\medskip
\subcaptionbox{Tail fits for $G$ and $H$.}
{\includegraphics[width=0.9\linewidth]{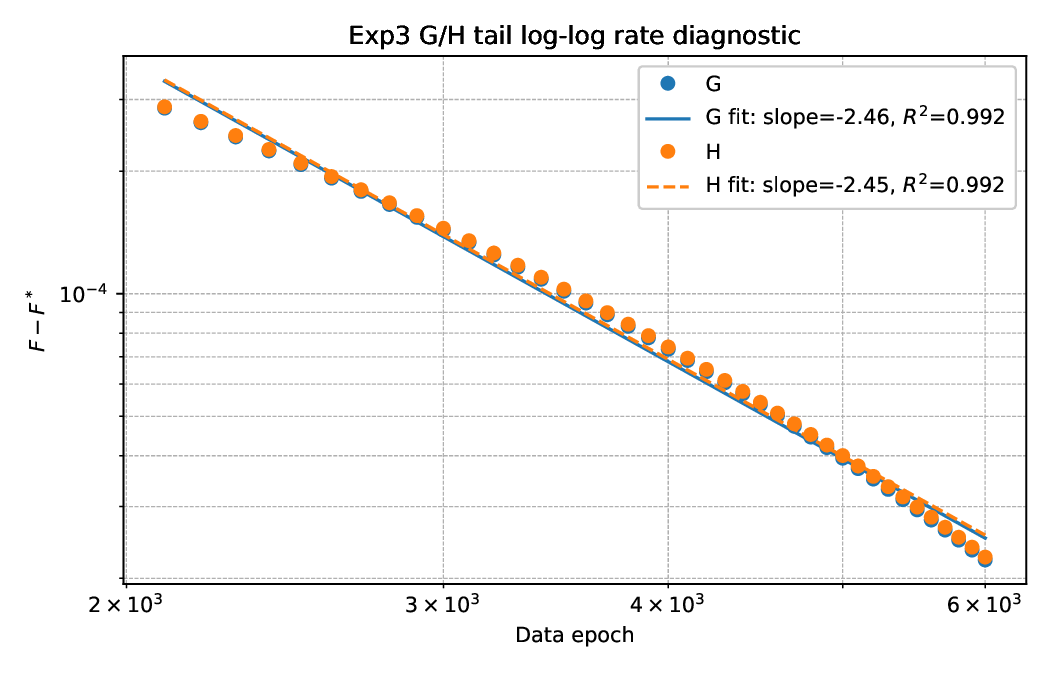}}
\caption{Pairwise objective. Panel (a) uses a logarithmic vertical axis and the method styles of Fig.~\ref{fig:exp-gram}. Panel (b) uses logarithmic axes and fits the mean gaps over the last $60\%$ of stored records, corresponding to data epochs $2100$--$6000$. The fitted slopes are $-2.46$ for $G$ and $-2.45$ for $H$, with $R^2\approx0.992$ for both. All horizontal coordinates count complete data epochs.}
\label{fig:exp-pairwise}
\end{figure}

Both naive methods reach a persistent nonzero plateau, whereas $G$ and $H$ continue decreasing. Their mean final gaps are $2.22\times10^{-5}$ and $2.26\times10^{-5}$, compared with approximately $0.143$ and $0.145$ for naive-raw and naive-q. All five runs of each of $G$ and $H$ first satisfy the recorded threshold $F-F^*\leq10^{-4}$ between epochs $3300$ and $4100$; neither naive method reaches it. These are endpoint measurements, with records every $100$ epochs after the initial densely sampled phase. Individual tail slopes lie between $-2.70$ and $-2.27$ for $G$, and between $-2.69$ and $-2.26$ for $H$. The decrease over this interval is faster than the $k^{-2}$ reference upper-bound regime, and its exponent is reported as a finite-window fit.

\begin{figure}[htbp]
\centering
\subcaptionbox{Gram objective.}{%
\begin{minipage}{0.72\linewidth}
\includegraphics[width=\linewidth,trim=40bp 0 372bp 15bp,clip]{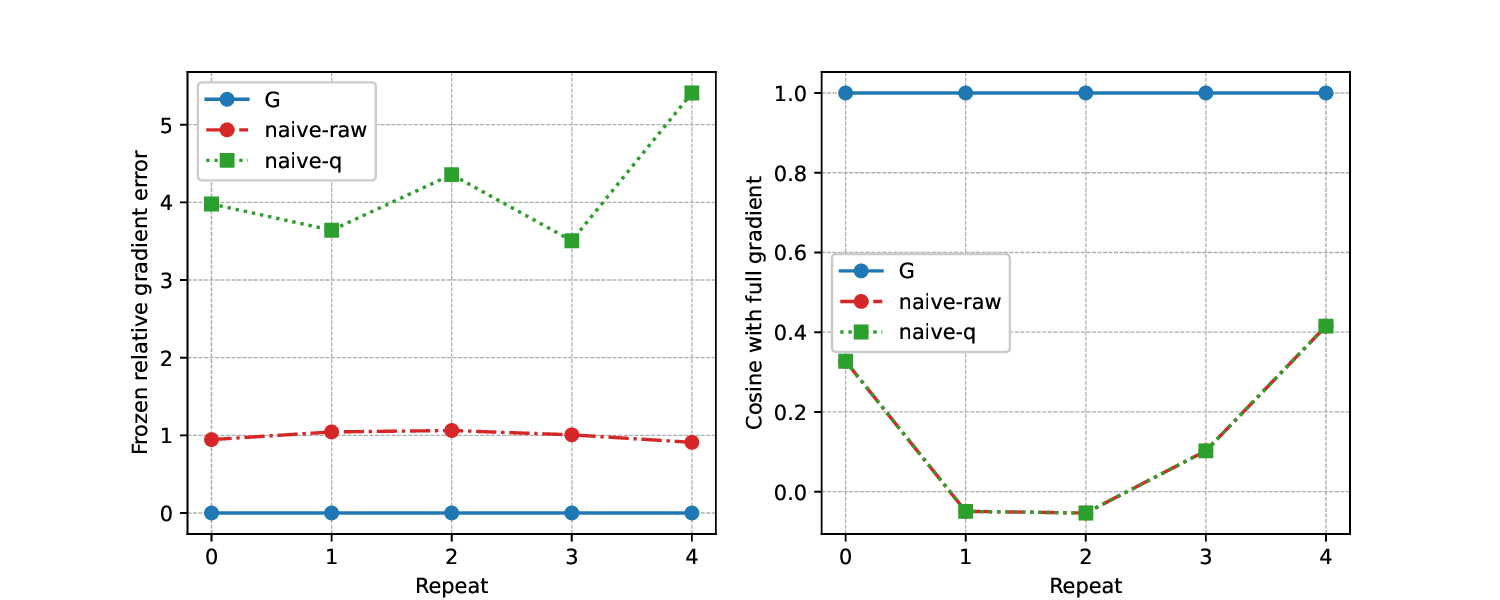}\par\medskip
\includegraphics[width=\linewidth,trim=344.364bp 0 67.636bp 15bp,clip]{fig/exp2_rank8_frozen.eps}
\end{minipage}}
\caption{Frozen parameter-gradient diagnostics at the five initializations. Each point uses one permutation with parameters held fixed throughout the epoch. In each panel, the upper plot shows relative errors against the full parameter gradient; the lower plot shows directional cosines. The naive-raw and naive-q cosines coincide because their frozen sums differ by the positive scalar $1/q$; their relative errors need not coincide. The pairwise $G$ and initially exact-cache $H$ sums agree with the full gradient to numerical precision and are omitted from panel (b).}
\label{fig:exp-frozen}
\end{figure}

\begin{figure}[htbp]
\ContinuedFloat
\centering
\subcaptionbox{Pairwise objective.}{%
\begin{minipage}{0.8\linewidth}
\includegraphics[width=\linewidth,trim=-17.28bp 0 403.94bp 0,clip]{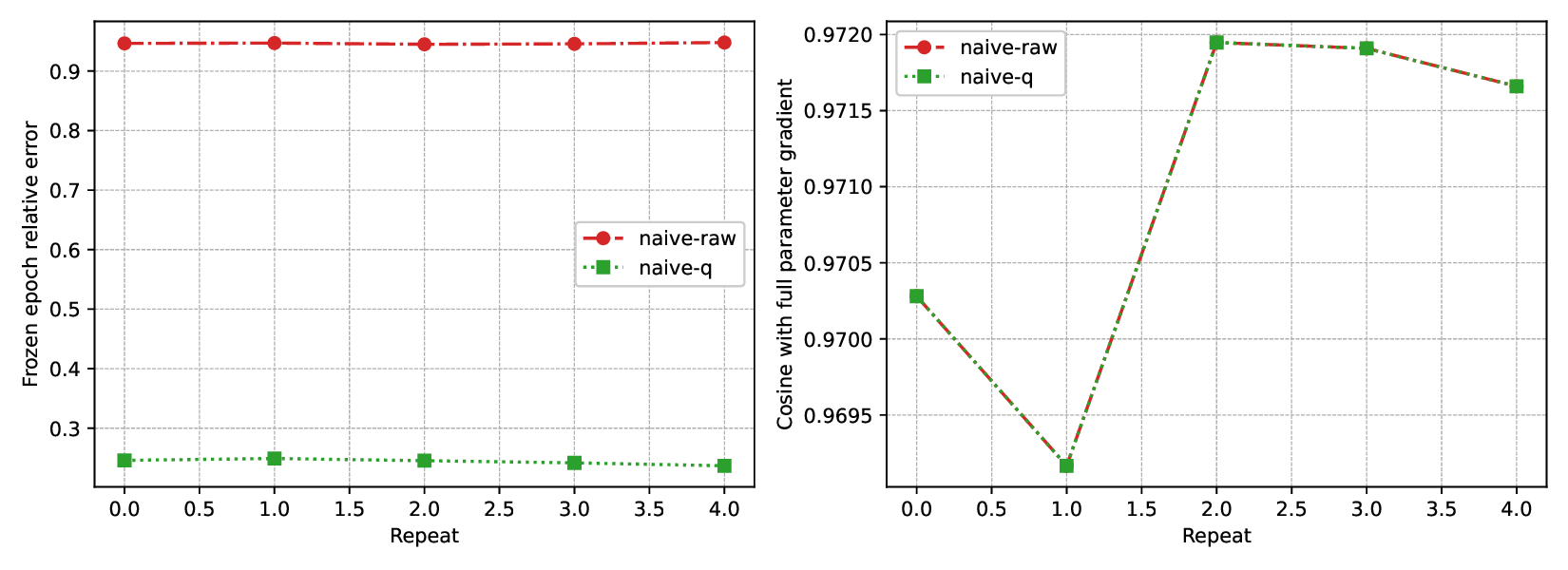}\par\medskip
\includegraphics[width=\linewidth,trim=382.86bp 0 3.8bp 0,clip]{fig/exp3_pairwise_frozen.eps}
\end{minipage}}
\caption[]{Frozen parameter-gradient diagnostics (continued).}
\end{figure}

\FloatBarrier
In the pairwise frozen diagnostic, naive-q reduces the relative error from approximately $0.945$--$0.948$ to $0.237$--$0.249$, while the common cosine of the two naive sums remains approximately $0.969$--$0.972$. Nevertheless, both online methods settle at a much larger full-objective gap than $G$ and $H$.

\subsection{Repeat variability and computational cost.}
Table~\ref{tab:exp-results} reports the final gaps and accumulated update-loop times. Timers include the gradient evaluations, forward and backward operations, optimizer steps, and cache refreshes within each batch update. Initialization, permutation generation, endpoint measurements, frozen diagnostics, saving, and plotting lie outside these timed regions. The experiments use CPU execution with eight threads configured for the numerical libraries.

\begin{table}[htbp]
\centering
\small
\begin{tabular}{llcc}
\hline
Objective & Method & Final $F-F^*$ & Time (s)\\
\hline
Gram & $G$ & $(5.570\pm0.425)\times10^{-3}$ & $320.8\pm16.0$\\
 & $H$ & $(5.622\pm0.427)\times10^{-3}$ & $233.8\pm14.2$\\
 & naive-raw & $(1.952\pm0.050)\times10^{-1}$ & $187.5\pm10.9$\\
 & naive-q & $(1.858\pm0.287)\times10^{-1}$ & $188.6\pm12.2$\\
\hline
Pairwise & $G$ & $(2.218\pm0.782)\times10^{-5}$ & $178.9\pm11.9$\\
 & $H$ & $(2.255\pm0.797)\times10^{-5}$ & $98.8\pm9.2$\\
 & naive-raw & $(1.4339\pm0.0039)\times10^{-1}$ & $80.5\pm8.5$\\
 & naive-q & $(1.4483\pm0.0124)\times10^{-1}$ & $80.2\pm7.9$\\
\hline
\end{tabular}
\caption{Final full-data gap and timed update-loop cost, reported as mean $\pm$ sample standard deviation over five runs. Each row uses the full budget of its experiment: $15{,}000$ data epochs for Gram and $6{,}000$ for pairwise.}
\label{tab:exp-results}
\end{table}

For the Gram problem, $H$ reduces the mean timed cost by approximately $27\%$ relative to $G$, while its mean final gap differs by less than $1\%$. For the pairwise problem, the corresponding time reduction is approximately $45\%$, with a final-gap difference below $2\%$. The naive updates are cheaper per fixed training budget but leave larger final full-data gaps in the two tested coupled objectives.

The persistent storage used by the cache can be counted directly. In double precision, the Gram implementation stores $8Nd=12{,}800$ bytes for $\widetilde Y$ and $8N^2=320{,}000$ bytes for $\widetilde K$, totaling $332{,}800$ bytes ($325$ KiB). The pairwise implementation stores $8N=1{,}600$ bytes for $\widetilde y$. The shared Gram target and input arrays occupy $320{,}000$ and $51{,}200$ bytes, respectively; the pairwise inputs and signs occupy $51{,}200$ and $1{,}600$ bytes. These counts describe stored arrays rather than total process memory, which also includes model parameters, differentiation buffers, and temporary arrays. In particular, the dense Gram cache trades additional persistent storage for fewer full-network evaluations; the pairwise gradient needs only the cached output vector and moment sums.

In the two tested coupled objectives, $G$ and $H$ attain lower final full-data gaps than the batch-local baselines, while $H$ attains gaps close to those of $G$ at a lower measured update-loop cost. The Gram example connects with the spectral setting of SpecNet2~\cite{chen2022specnet2}, while the pairwise example exhibits the same estimator distinction without a Gram objective. At a stationary output, every ideal batch contribution vanishes even when $F^*>0$, as in both coupled constructions here. The relevant distinction is therefore the preservation of the global gradient balance under batch restriction, rather than a positive minimum value by itself.

\section{Conclusion}
\par\nopagebreak[4]
We study a class of smooth parameterized non-separable objectives $\mathcal L(\theta)=F(Y(\theta))$ and analyze stochastic updates that preserve the full output-gradient structure under mini-batch sampling. The ideal row-sampled estimator has a deterministic batch scaling and a second moment controlled by the full output gradient, which supports fixed positive step sizes within the local theory. The resulting stopping-time and KL analysis gives quantitative confinement bounds and exponent-dependent linear or polynomial rates on the non-exit event.

Parameter updates can change unselected outputs and thereby introduce cache-tracking error. The Lyapunov analysis of the cached surrogate couples this error with the objective gap and yields the corresponding local high-probability convergence and KL rates. These results clarify the role of full-output gradient sampling and cache tracking in fixed-step stochastic optimization for coupled objectives. Future work includes sharper batch-size-dependent constants, broader KL regions, and extensions to other stochastic optimization schemes.

\section*{Declarations}

\noindent\textbf{Funding} This work was supported in part by the National Natural Science Foundation of China under Grant Nos. 12271109 and 12526211; by the Shanghai Pilot Program for Basic Research-Fudan University under Grant No. 21TQ1400100 (22TQ017); by the Scientific Research Innovation Capability Support Project for Young Faculty under Grant No. SRICSPYF-ZY2025159; and by the Xuemin Institute of Advanced Studies, Fudan University.

\noindent\textbf{Competing interests} The authors declare that they have no conflict of interest.

\noindent\textbf{Data availability} The source code for the numerical experiments is available at \url{https://github.com/wuruofan123/nonseparable-sgd}.

\appendix
\raggedbottom
\section{Proofs and Detailed Derivations}
\label{app:detailed_proofs}
\subsection{Proof of Lemma \ref{l31}.}
\label{d_l31}
\begin{proof}
    By the Cauchy--Schwarz inequality,
    \begin{equation*}
    \begin{aligned}
        \|G\|_{\F}^{2}
        &= \Bigl\|\sum_{i\in\mathcal{B}} z_{i}\Bigr\|_{\F}^{2}
         = \Bigl\|\sum_{i\in\mathcal{B}}(\nabla_{Y}F)_{i}\cdot\nabla_{\theta}Y_{i}\Bigr\|_{\F}^{2} \\
        &\leq \Bigl(\sum_{i\in\mathcal{B}} \|(\nabla_{Y}F)_{i}\cdot\nabla_{\theta}Y_{i}\|_{\F}\Bigr)^{2} \\
        &\leq \Bigl(\sum_{i\in\mathcal{B}} \|(\nabla_{Y}F)_{i}\|_{2}\,\|\nabla_{\theta}Y_{i}\|_{\F}\Bigr)^{2} \\
        &\leq \Bigl(\sum_{i\in\mathcal{B}}\|(\nabla_{Y}F)_{i}\|_{2}^{2}\Bigr)
           \Bigl(\sum_{i\in\mathcal{B}}\|\nabla_{\theta}Y_{i}\|_{\F}^{2}\Bigr) \\
        &\leq B_{1}^{2}\,\|\nabla_{Y}F\|_{\F}^{2}.
    \end{aligned}
\end{equation*}
For the conditional second moment, use the inclusion indicators and the fact that each row is selected with probability $p$:
\[
\E[\|G_k\|_{\F}^{2}\mid\mathcal F_k]
\leq B_1^2\E\left[\sum_{i\in\mathcal B_k}\|g_{i,k}\|_2^2\middle|\mathcal F_k\right]
=pB_1^2\|g_k\|_{\F}^{2}.
\]
For the first moment, the triangle inequality followed by the same inclusion probabilities gives
\[
\E[\|G_k\|_{\F}\mid\mathcal F_k]
\leq p\sum_{i=1}^N\|g_{i,k}\|_2\|\nabla_\theta Y_{i,k}\|_{\F}
\leq pB_1\|g_k\|_{\F},
\]
where the last step is Cauchy--Schwarz over the sample index. No independence between the batch inclusion indicators is required.
\end{proof}
\subsection{Proof of Lemma \ref{l32}.}
\label{d_l32}
\begin{proof}
By the Taylor expansion with integral remainder, we have
\begin{equation*}
    Y^{+} - Y= \nabla_{\theta}Y(\theta)(\theta^{+}-\theta) + R
\end{equation*} 
where 
\begin{equation*}
R = \int_{0}^{1}(1-t)\nabla^{2}_{\theta}Y(\theta+t(\theta^{+}-\theta))[\theta^{+}-\theta,\theta^{+}-\theta]\,\mathrm{d}t,
\end{equation*} 
and $\nabla_\theta^2 Y$ is the second-order derivative tensor of $Y$ with respect to $\theta$. 
Then, by taking the Frobenius norm and Assumption $\text{A}_4$ on the line segment between $\theta$ and $\theta^+$, we obtain
\begin{equation*}
    \|R\|_{\F}\leq \frac{B_{2}}{2}\|\theta^{+}-\theta\|_{\F}^{2} = \frac{B_{2}}{2}\cdot\eta^{2}\|G\|_{\F}^{2}.
\end{equation*}   
Similarly, for the first-order difference, we have
\begin{equation*}
\|Y^{+} - Y\|_{\F}\leq \sup\limits_{t\in[0,1]}\|\nabla_{\theta}Y(\theta+t(\theta^{+}-\theta))\|_{\F}\cdot\|\theta^{+}-\theta\|_{\F}\leq B_{1} \cdot\eta\|G\|_{\F}. 
\end{equation*}
Since $\nabla_Y F$ is $L$‑Lipschitz by Assumption $\text{A}_2$ (i.e. $F$ is $L$‑smooth in $Y$), 
we have the standard quadratic upper bound
\begin{equation*}
    F(Y^{+}) \ \leq \ F(Y) + \langle\nabla_{Y} F, Y^{+} - Y \rangle
+ \frac{L}{2}\|Y^{+} - Y\|_{\F}^{2}.
\end{equation*}
Substituting the per‑row decomposition and summing over $i=1,\cdots,N$ gives
\begin{equation*}\label{eq1}
\begin{aligned}
F(Y^{+}) \ \leq & \ F(Y) + \langle\nabla_{Y} F, Y^{+} - Y \rangle
+ \frac{L}{2}\|Y^{+} - Y\|_{\F}^{2}\\
= &\ \ F(Y) + \tr\left((Y^{+} - Y)^{\T}(\nabla_{Y} F)\right) + \frac{L}{2}\|Y^{+} - Y\|_{\F}^{2}\\
= &\ \ F(Y) + \tr\left((\nabla_{\theta}Y(\theta)(\theta^{+}-\theta) + R)^{\T}(\nabla_{Y} F)\right) + \frac{L}{2}\|Y^{+} - Y\|_{\F}^{2}\\
= &\ F(Y) + \tr\left(\sum\limits_{i=1}^{N}\nabla_{\theta}Y_{i}(\theta)(\theta^{+}-\theta)(\nabla_{Y} F)_{i}\right)\\
&\quad + \langle \nabla_{Y} F, R \rangle + \frac{L}{2}\|Y^{+} - Y\|_{\F}^{2}\\
\leq & \ F(Y) - \eta\cdot\tr\big(\sum\limits_{i=1}^{N}\nabla_{\theta}Y_{i}(\theta)\cdot G^{\T}(\nabla_{Y} F)_{i}\big)\\
&\quad + \left(\|\nabla_{Y} F\|_{\F}\cdot\frac{B_{2}}{2}+\frac{L}{2}B_{1}^{2}\right)\cdot\eta^{2} \|G\|_{\F}^{2}\\
= & \ F(Y) - \eta\cdot\sum\limits_{i=1}^{N}\tr\big(\left((\nabla_{Y} F)_{i}\cdot\nabla_{\theta}Y_{i}\right)\cdot G^{\T}\big)\\
&\quad + \left(\|\nabla_{Y} F\|_{\F}\cdot\frac{B_{2}}{2}+\frac{L}{2}B_{1}^{2}\right)\cdot\eta^{2} \|G\|_{\F}^{2}\\
\leq &\ F(Y) - \eta\cdot\sum\limits_{i=1}^{N} \langle  z_{i}
,G^{\T}\rangle  + \frac{1}{2}\left({B_{0}B_{2}}+LB_{1}^{2}\right)\cdot\eta^{2}\|G\|_{\F}^{2}\\
\leq & \ F(Y) - \eta \cdot\sum\limits_{i=1}^{N}\langle  z_{i}
,G^{\T}\rangle  + \frac{1}{2}\left({B_{0}B_{2}}+LB_{1}^{2}\right)c_{1}\cdot\eta^{2}\|\nabla_{Y} F\|_{\F}^{2},
\end{aligned}
\end{equation*}
where we denote $c_{1}= B_{1}^2$ and then $\|G\|_{\F}^{2}\leq c_{1}\|\nabla_{Y}F\|_{\F}^{2}$.
Let $c_{2}= \frac{1}{2}\left({B_{0}B_{2}}+LB_{1}^{2}\right)c_{1}$. Then 
$$ F(Y^{+})\leq F(Y)-\eta\cdot\sum\limits_{i=1}^{N}\langle z_{i}, G^{\T}\rangle + c_{2}\cdot\eta^{2}\|\nabla_{Y}F\|_{\F}^{2}.$$
\end{proof}
\subsection{Proof of Lemma \ref{l33}.}
\label{d_l33}
\begin{proof}
    Starting from the intermediate bound in Lemma~\ref{l32},
    \[
    F(Y^{+})\leq F(Y)-\eta\sum_{i=1}^{N}\langle z_i,G^{\T}\rangle+a\eta^2\|G\|_{\F}^{2},
    \qquad a=\frac12(B_0B_2+LB_1^2),
    \]
    take conditional expectation with respect to the current mini-batch $\mathcal B$. Since each index is included with probability $p=b/N$,
    \[
    \E_{\mathcal B}[G^{\T}\mid\mathcal F_k]=p\sum_{j=1}^{N}z_j,
    \qquad
    \E_{\mathcal B}[\|G\|_{\F}^{2}\mid\mathcal F_k]
    \le pB_1^2\|\nabla_YF\|_{\F}^{2}.
    \]
    \begin{equation*}
    \begin{aligned}
    \E_{\mathcal B}[F(Y^{+})\mid\mathcal F_k]
    &\leq F(Y)-p\eta\|J^{\T}g\|^2+paB_1^2\eta^2\|\nabla_YF\|_{\F}^{2}\\
    &\leq F(Y)-p\eta(\lambda_{\min}-c_2\eta)\|\nabla_YF\|_{\F}^{2}.
    \end{aligned}
    \end{equation*}
    where $c_2=aB_1^2$ and $c_3:=p(\lambda_{\min}-c_2\eta)>0$. If $c_2=0$, the positivity condition imposes no upper bound through $c_2$.
    If $k<\tau_G$, the margin condition places the complete update segment, including a possible exit endpoint, in $Q_{\mathrm{out}}$; applying this inequality conditionally gives
    \[
    \E[Z_{k+1}\mid\mathcal F_k]\leq Z_k-c_3\eta\|g_k\|_{\F}^{2}.
    \]
    If $k\geq\tau_G$, the stopped process is already frozen, so $Z_{k+1}=Z_k$. Combining the two cases proves
    \[
    \E[Z_{k+1}\mid\mathcal F_k]\leq Z_k-c_3\eta\|g_k\|_{\F}^{2}\mathbf 1_{\{k<\tau_G\}}.
    \]
\end{proof}
\subsection{Detailed Derivation for Remark \ref{r31}.}
\label{d_r31}
    Summing the stopped supermartingale inequality and taking expectations gives
    \begin{equation*}
    c_3\eta\sum_{k=0}^{T}\E\left[\|g_k\|_{\F}^{2}\mathbf1_{\{k<\tau_G\}}\right]
    \leq f_0-\E[Z_{T+1}]
    \leq f_0.
    \end{equation*}
    The minimum bound follows by dividing by $T+1$. Finally, Cauchy--Schwarz gives
    \[
    \min_{0\leq k\leq T}\E\left[\|g_k\|_{\F}\mathbf1_{\{k<\tau_G\}}\right]
    \leq \sqrt{\min_{0\leq k\leq T}\E\left[\|g_k\|_{\F}^{2}\mathbf1_{\{k<\tau_G\}}\right]}
    \leq\sqrt{\frac{f_0}{c_3\eta(T+1)}}.
    \]
\subsection{Proof of Theorem \ref{t34}.}
\label{d_t34}
\begin{proof}
Put $A_k=\{k<\tau_G\}$ and $X_k:=Z_k\mathbf1_{A_k}$. Since
$A_{k+1}\subset A_k$, Lemma~\ref{l33} and nonnegativity give
\begin{equation}\label{eq:t34-killed-one-step}
 \E[X_{k+1}\mid\mathcal F_k]
 \leq X_k-c_3\eta\|g_k\|_{\F}^2\mathbf1_{A_k}.
\end{equation}
This is a genuinely killed inequality; no value of the original process after
its first exit is used.

For $f_k>0$, concavity of $\Phi$ and \eqref{eq:t34-killed-one-step} yield,
on $A_k$,
\begin{align*}
 \E[\Phi(X_{k+1})\mid\mathcal F_k]
 &\leq \Phi(f_k)-c_3\eta\Phi'(f_k)\|g_k\|_{\F}^2\\
 &\leq \Phi(f_k)-c_3\eta\|g_k\|_{\F},
\end{align*}
where $\Phi'(f_k)=f_k^{-\alpha}/c$ and the KL condition was used in the last
inequality. If $f_k=0$, differentiability and $f\geq0$ on the open set $D$
imply $g_k=0$; then (\ref{eq:t34-killed-one-step}) forces the next killed
gap to be zero, so no derivative at zero is needed. Telescoping gives
\begin{equation}\label{eq:t34-length}
 \E\sum_{k<\tau_G}\eta\|g_k\|_{\F}\leq\frac{\Phi(f_0)}{c_3}.
\end{equation}
The estimate includes the step that reaches the boundary. By Lemma~\ref{l31},
its parameter increment is at most controlled in conditional expectation by
\[
\E[\|\theta_{k+1}-\theta_k\|\mid\mathcal F_k]
\leq pB_1\eta\|g_k\|_{\F},
\]
hence
\[
 R\,\mathbb P(\tau_G<\infty)
 \leq\E\sum_{k<\tau_G}\|\theta_{k+1}-\theta_k\|
 \leq\frac{pB_1\Phi(f_0)}{c_3}.
\]
Markov's inequality therefore gives $\mathbb P(E_G)\geq1-\zeta$ under
the stated margin condition.

On $A_k$, the KL condition and (\ref{eq:t34-killed-one-step}) give
\[
 \E[X_{k+1}\mid\mathcal F_k]\leq X_k-d_GX_k^{2\alpha},
 \qquad d_G=c_3\eta c^2.
\]
Jensen's inequality, applied to the killed variable $X_k$ (not to an
unkilled or future-event-conditioned variable), gives
\[
 m_{k+1}\leq m_k-d_Gm_k^{2\alpha},
 \qquad m_k=\E X_k.
\]
For $\alpha=1/2$, $d_G<1$ and iteration give
$m_k\leq f_0(1-d_G)^k$. In every case the recurrence forces $m_k\to0$,
since a positive limit would incur a fixed positive decrement. Because
$E_G\subset A_k$,
\[
 \E[Z_k\mathbf1_{E_G}]\leq m_k,
 \qquad \E[Z_k\mid E_G]\leq\frac{m_k}{1-\zeta}.
\]
The nonnegative supermartingale $(X_k)$ therefore converges almost surely to a finite limit $X_\infty$. Since $E_G\subset A_k$ for every $k$,
\[
\mathbb E[X_k\mathbf1_{E_G}]=\mathbb E[Z_k\mathbf1_{E_G}]\leq m_k\longrightarrow0.
\]
Fatou's lemma now gives
\[
\mathbb E[X_\infty\mathbf1_{E_G}]
\leq \liminf_{k\to\infty}\mathbb E[X_k\mathbf1_{E_G}]=0.
\]
Hence $X_\infty=0$ almost surely on $E_G$. On that event $X_k=Z_k=f_k$ for every $k$, so $Z_k\to0$ almost surely on $E_G$. Equation
(\ref{eq:t34-length}) and the increment bound imply finite total parameter
length almost surely there. Thus the iterates are Cauchy on $E_G$ and
converge to $\theta_\infty\in\overline{Q_{\mathrm{in}}}$. Continuity of $F\circ Y$ gives
$f(Y(\theta_\infty))=0$.
\end{proof}
\subsection{Proof of Corollary \ref{coro35}.}
\label{d_coro35}
\begin{proof}
Set $X_k=Z_k\mathbf1_{\{k<\tau_G\}}$. Under $\alpha=1/2$, Theorem~\ref{t34}
gives $\E[Z_k\mathbf1_{E_G}]\leq f_0\rho_G^k$. Fix
$\widehat\rho\in(\rho_G,1)$ and set
$A_k=E_G\cap\{Z_k>\widehat\rho^k\}$. Markov's inequality applied to the event $A_k$ gives
\[
 \mathbb P(A_k)\leq f_0(\rho_G/\widehat\rho)^k.
\]
The right-hand side is summable, so Borel--Cantelli implies that only
finitely many $A_k$ occur. On $E_G$, $Z_k=f_k$; hence almost every path
in $E_G$ satisfies $f_k\leq\widehat\rho^k$ for all sufficiently large $k$.
\end{proof}
\subsection{Proof of Lemma \ref{l41}.}
\label{d_l41}
\begin{proof}
Recall $H=\sum_{i\in\mathcal{B}}p_i$ with $p_i=(\widetilde{\nabla_YF})_i\cdot\nabla_\theta Y_i$.
By Cauchy--Schwarz,
\[
\begin{aligned}
\|H\|_\F^2
&\leq\left(\sum_{i\in\mathcal B}\|p_i\|_\F\right)^2
\leq\left(\sum_{i\in\mathcal B}\|
(\widetilde{\nabla_YF})_i\|_2^2\right)
\left(\sum_{i\in\mathcal B}\|\nabla_\theta Y_i\|_\F^2\right)\\
&\leq B_1^2\|\widetilde{\nabla_YF}\|_\F^2.
\end{aligned}
\]
\end{proof}

\subsection{Proof of Theorem \ref{t36}.}
\label{d_t36}
\begin{proof}
Let $X_k=Z_k\mathbf 1_{\{k<\tau_G\}}$ and $m_k=\E X_k$. Lemma~\ref{l33} and the inclusion of survival events give
\[
m_{k+1}\leq m_k-c_3\eta\,\E[\|g_k\|_\F^2\mathbf1_{\{k<\tau_G\}}].
\]
By the KL condition, $\|g_k\|_\F\geq c f_k^\alpha$ on $\{k<\tau_G\}$, so
\[
\E[\|g_k\|_\F^2\mathbf1_{\{k<\tau_G\}}]
\geq c^2\E[X_k^{2\alpha}]
\geq c^2m_k^{2\alpha}
\]
by Jensen. Thus $m_{k+1}\leq m_k-d_Gm_k^{2\alpha}$, where $d_G=c_3\eta c^2$.

For $f_0>0$, set $t=2\alpha-1>0$. Whenever $m_{k+1}>0$, convexity of $x^{-t}$ gives
\[
m_{k+1}^{-t}-m_k^{-t}\geq t m_k^{-t-1}(m_k-m_{k+1})\geq t d_G.
\]
If some $m_j=0$, nonnegativity and the recurrence imply $m_k=0$ for all $k\geq j$; otherwise, summing gives
\[
m_k\leq\left[f_0^{-(2\alpha-1)}+(2\alpha-1)d_Gk\right]^{-1/(2\alpha-1)}.
\]
Since $E_G\subset\{k<\tau_G\}$, the event-restricted bounds are
\[
\E[Z_k\mathbf1_{E_G}]\leq m_k,
\qquad \E[Z_k\mid E_G]\leq\frac{m_k}{1-\zeta}.
\]
No claim is made outside $E_G$. If $f_0=0$, then $g_0=0$ because $f\geq0$ on the open set $D$; the update is fixed and all gaps are zero.
\end{proof}

\subsection{Proof of Corollary \ref{coro37}.}
\label{d_coro37}
\begin{proof}
Let $\gamma=(2\alpha-1)^{-1}$ and $X_k=Z_k\mathbf1_{\{k<\tau_G\}}$. Theorem~\ref{t36} gives
\[
\E X_k=O(k^{-\gamma}).
\]
The killed process $(X_k)$ is a nonnegative supermartingale by Lemma~\ref{l33}. Fatou's lemma yields finite $\liminf_k k^\gamma X_k$ almost surely, and on $E_G$, Theorem~\ref{t34} gives $f_k\to0$. Enumerating infinitely many indices where $k^\gamma f_k$ is bounded proves the stated random subsequence.

Fix $\epsilon\in(0,\gamma)$ and put $a=\gamma-\epsilon>0$. For every integer $j\geq1$, the finite-horizon maximal inequality for the nonnegative supermartingale $(X_k)_{k\geq 2^j}$, followed by monotone convergence as the horizon tends to infinity, gives
\[
\begin{aligned}
\Pr\left(\sup_{k\geq 2^j}X_k>2^{-ja}\right)
&\leq 2^{ja}\E X_{2^j}
\leq C2^{-j\epsilon}
\end{aligned}
\]
for a deterministic constant $C$. The right-hand side is summable in $j$, so Borel--Cantelli implies that, almost surely, for all sufficiently large $j$,
\[
\sup_{k\geq2^j}X_k\leq2^{-ja}.
\]
If $2^j\leq k<2^{j+1}$, this gives $X_k\leq2^a k^{-a}$. Hence
\[
X_k=O_\omega\!\left(k^{-\gamma+\epsilon}\right)\quad\text{almost surely}.
\]
On $E_G$, $X_k=f_k$, which proves the full-sequence assertion. All conclusions are restricted to $E_G$.
\end{proof}

\subsection{Detailed Derivation for Remark \ref{r32}.}
\label{d_r32}
\begin{proof}
Assume the convex auxiliary setting of Remark~\ref{r32}, with $\|Y(\theta)-Y^*\|_\F\leq B$ on $Q_{\rm out}$. Convexity and $\nabla F(Y^*)=0$ give $f\leq B\|g\|_\F$ before exit. For $X_k=Z_k\mathbf1_{\{k<\tau_G\}}$, the stopped inequality and Jensen yield $m_{k+1}\leq m_k-(c_3\eta/B^2)m_k^2$, hence $m_k\leq B^2/(c_3\eta k)$ for $k\ge1$. This is only a killed estimate and supplies no confinement probability without the KL condition.

Under Theorem~\ref{t36}, for $\epsilon\in(0,\gamma)$ take $k_n=\lfloor n^{2/\epsilon}\rfloor+n$. Markov's inequality applied to $E_G\cap\{Z_{k_n}>k_n^{-\gamma+\epsilon}\}$, together with $\E[Z_{k_n}\mathbf1_{E_G}]=O(k_n^{-\gamma})$, gives summable probabilities; Borel--Cantelli gives the stated sparse rate on $E_G$.
\end{proof}

\subsection{Proof of Lemma \ref{l42}.}
\label{d_l42}
\begin{proof}
By the Taylor expansion with integral remainder, we have
\begin{equation*}
    Y^{+} - Y= \nabla_{\theta}Y(\theta)(\theta^{+}-\theta) + R
\end{equation*} 
where 
\begin{equation*}
R = \int_{0}^{1}(1-t)\nabla^{2}_{\theta}Y(\theta+t(\theta^{+}-\theta))[\theta^{+}-\theta,\theta^{+}-\theta]\,\mathrm{d}t,
\end{equation*} 
and $\nabla_\theta^2 Y$ is the second-order derivative tensor of $Y$ with respect to $\theta$. 
Then, by taking the Frobenius norm and Assumption $\text{A}_4$ on the line segment between $\theta$ and $\theta^+$, we obtain
\begin{equation*}
    \|R\|_{\F}\leq \frac{B_{2}}{2}\|\theta^{+}-\theta\|_{\F}^{2} = \frac{B_{2}}{2}\cdot\eta^{2}\|H\|_{\F}^{2}.
\end{equation*}   
Similarly, for the first-order difference, we have
\begin{equation*}
\|Y^{+} - Y\|_{\F}\leq \sup\limits_{t\in[0,1]}\|\nabla_{\theta}Y(\theta+t(\theta^{+}-\theta))\|_{\F}\cdot\|\theta^{+}-\theta\|_{\F}\leq B_{1} \cdot\eta\|H\|_{\F}. 
\end{equation*}
Since $\nabla_Y F$ is $L$‑Lipschitz by Assumption $\text{A}_2$ (i.e. $F$ is $L$‑smooth in $Y$), 
we have the standard quadratic upper bound
\begin{equation*}
    F(Y^{+}) \ \leq \ F(Y) + \langle\nabla_{Y} F, Y^{+} - Y \rangle
+ \frac{L}{2}\|Y^{+} - Y\|_{\F}^{2}.
\end{equation*}
Substituting the per‑row decomposition and summing over $i=1,\cdots,N$ gives
\begin{equation*}
\begin{aligned}
F(Y^{+}) \ \leq & \ F(Y) + \langle\nabla_{Y} F, Y^{+} - Y \rangle
+ \frac{L}{2}\|Y^{+} - Y\|_{\F}^{2}\\
= &\ \ F(Y) + \tr\left((\nabla_{\theta}Y(\theta)(\theta^{+}-\theta) + R)^{\T}(\nabla_{Y} F)\right) + \frac{L}{2}\|Y^{+} - Y\|_{\F}^{2}\\
= &\ F(Y) + \tr\left(\sum\limits_{i=1}^{N}\nabla_{\theta}Y_{i}(\theta)(\theta^{+}-\theta)(\nabla_{Y} F)_{i}\right)\\
&\quad + \langle \nabla_{Y} F, R \rangle + \frac{L}{2}\|Y^{+} - Y\|_{\F}^{2}\\
\leq & \ F(Y) - \eta\cdot\tr\big(\sum\limits_{i=1}^{N}\nabla_{\theta}Y_{i}(\theta)\cdot H^{\T}(\nabla_{Y} F)_{i}\big)\\
&\quad + \left(\|\nabla_{Y} F\|_{\F}\cdot\frac{B_{2}}{2}+\frac{L}{2}B_{1}^{2}\right)\cdot\eta^{2} \|H\|_{\F}^{2}\\
= & \ F(Y) - \eta\cdot\sum\limits_{i=1}^{N}\tr\big(\left((\nabla_{Y} F)_{i}\cdot\nabla_{\theta}Y_{i}\right)\cdot H^{\T}\big)\\
&\quad + \left(\|\nabla_{Y} F\|_{\F}\cdot\frac{B_{2}}{2}+\frac{L}{2}B_{1}^{2}\right)\cdot\eta^{2} \|H\|_{\F}^{2}\\
\leq & \ F(Y) - \eta\cdot\sum\limits_{i=1}^{N} \langle  z_{i}
,H^{\T}\rangle  + \frac{1}{2}\left({B_{0}B_{2}}+LB_{1}^{2}\right)\cdot\eta^{2}\|H\|_{\F}^{2}\\
\leq & \ F(Y) - \eta \cdot\sum\limits_{i=1}^{N}\langle  z_{i}
,H^{\T}\rangle  + c_{2}\cdot\eta^{2}\|\widetilde{\nabla_{Y} F}\|_{\F}^{2},
\end{aligned}
\end{equation*}
 where $z_{i}=\left((\nabla_{Y}F)_{i}\cdot\nabla_{\theta}Y_{i}\right)^{\T}$ is as defined in Section~\ref{subsec:notation}.
\end{proof}
\subsection{Proof of Lemma \ref{l43}.}
\label{d_l43}
\begin{proof}
We sample a mini-batch $\mathcal{B}$ for the current update, which determines the batch gradient $H$ and, in turn, $\tilde{Y}^{+}$ and $Y^{+}$, while $\tilde{Y}$ and $Y$ remain fixed from the preceding step.
Using the Frobenius-norm representation,
$$\|\tilde{Y}^{+}-Y^{+}\|_{\F}^{2} = \sum\limits_{i\in \mathcal{B}^{c}}\|\tilde{Y}_{i}^{+}-Y_{i}^{+}\|_{\F}^{2}.$$
For each $i\in\mathcal{B}^{c}$, we decompose
$$\tilde{Y}_{i}^{+}-Y_{i}^{+}= \tilde{Y}_{i}-Y_{i}^{+}=(\tilde{Y}_{i}-Y_{i})+(Y_{i}-Y_{i}^{+}).$$
Expanding the squared norm and applying the inequality $2\langle a,b\rangle \leq \frac{1}{u}\|a\|_{2}^{2} + u\|b\|_{2}^{2}$ (valid for any $u>0$) yields
$$ \|\tilde{Y}_{i}-Y_{i}^{+}\|_{2}^{2}\leq (1+\frac{1}{u})\|\tilde{Y}_{i}-Y_{i}\|_{2}^{2}+(1+u)\|Y_{i}-Y_{i}^{+}\|_{2}^{2}.$$
Let $\mathcal{F}$ denote the filtration generated by all randomness up to the current iterate. By the tower property, the total expectation decomposes as 
$$\E[\cdot]=\E_{\mathcal{F}}\Bigl[\E_{\mathcal{B}}\bigl[\cdot\mid\mathcal{F}\bigr]\Bigr],$$ 
where the inner expectation is taken over the random mini-batch $\mathcal{B}$. For the single-step analysis, we focus on this inner expectation. Summing over $i\in\mathcal{B}^{c}$ and taking conditional expectation given $\mathcal{F}$ yields
$$\begin{aligned}
\E_{\mathcal{B}}\!\left[\|\tilde{Y}^{+}-Y^{+}\|_{\F}^{2}\;\middle|\;\mathcal{F}\right]\leq & \left(1+\frac{1}{u}\right)\E_{\mathcal{B}}\!\left[\sum\limits_{i\in \mathcal{B}^{c}}\|\tilde{Y}_{i}-Y_{i}\|_{\F}^{2}\;\middle|\;\mathcal{F}\right]\\
&+\left(1+u\right)\E_{\mathcal{B}}\!\left[\sum\limits_{i\in \mathcal{B}^{c}}\|Y_{i}-Y_{i}^{+}\|_{\F}^{2}\;\middle|\;\mathcal{F}\right].
\end{aligned}$$
In the remainder of this section, we abbreviate $\E_{\mathcal{B}}[\cdot\mid\mathcal{F}]$ as $\E_{\mathcal{B}}[\cdot]$.
Because $\tilde{Y}_i-Y_i$ is independent of the random batch $\mathcal{B}$, every index $i\in\{1,\dots,N\}$ is excluded from $\mathcal{B}$ with the same probability $1-\frac{|\mathcal{B}|}{N}$. By linearity of expectation,
$$\E_{\mathcal{B}}\big[\sum\limits_{i\in \mathcal{B}^{c}}\|\tilde{Y}_{i}-Y_{i}\|_{\F}^{2}\big] = (1-\frac{|\mathcal{B}|}{N})\sum\limits_{i=1}^{N}\|\tilde{Y}_{i}-Y_{i}\|_{\F}^{2}=(1-\frac{|\mathcal{B}|}{N})\|\tilde{Y}-Y\|_{\F}^{2}.$$
For the second term note that summing over $\mathcal{B}^{c}$ is bounded by summing over all samples:
$$\E_{\mathcal{B}}\big[\sum\limits_{i\in \mathcal{B}^{c}}\|Y_{i}-Y_{i}^{+}\|_{\F}^{2}\big]\leq \E_{\mathcal{B}}\big[\sum\limits_{i=1}^{N}\|Y_{i}-Y_{i}^{+}\|_{\F}^{2}\big]=
\E_{\mathcal{B}}\big[\|Y-Y^{+}\|_{\F}^{2}\big].$$
Using the update $\theta^{+}-\theta=-\eta H^{T}$ and $\|\nabla_{\theta}Y\|_{\F}\leq B_{1}$ (Assumption $\text{A}_{3}$), we have
$$\|Y-Y^{+}\|_{\F}\leq B_{1}\|\theta-\theta^{+}\|_{\F} = B_{1}\eta\|H\|_{\F}.$$
Applying Lemma \ref{l41} ($\|H\|_{\F}\leq B_1\|\widetilde{\nabla_{Y}F}\|_{\F}$) yields
$$\E_{\mathcal{B}}\big[\|Y-Y^{+}\|_{\F}^{2}\big]\leq \eta^{2}B_{1}^{2}\E_{\mathcal{B}}\big[\|H\|_{\F}^{2}\big]\leq \eta^{2}B_{1}^{4}\|\widetilde{\nabla_{Y}F}\|_{\F}^{2}.$$
Combining the preceding bounds gives the stated inequality
$$ \E_{\mathcal{B}}\big[\|\tilde{Y}^{+}-Y^{+}\|_{\F}^{2}\big]\leq \left(1+\frac{1}{u}\right)\left(1-\frac{|\mathcal{B}|}{N}\right)\|\tilde{Y}-Y\|_{\F}^{2} +
        (1+u)\cdot\eta^{2} B_{1}^{4}\|\widetilde{\nabla_{Y}F}\|_{\F}^{2}.$$
\end{proof}
\subsection{Proof of Lemma \ref{l44}.}
\label{d_l44}
\begin{proof}
Fix a current state with $\theta_k\in Q_{\mathrm{in}}$ and $\|e_k\|_{\F}\leq r$, and condition throughout on $\mathcal{F}_k$. We first establish
\begin{equation}\label{eq:H_gap_descent}
 \E[f_{k+1}\mid\mathcal{F}_k]
 \leq f_k-\frac{p\lambda\eta}{2}\|g_k\|_{\F}^2
       +D_H\eta\|e_k\|_{\F}^2+c_2\eta^2\|h_k\|_{\F}^2.
\end{equation}
This follows from the Jacobian lower bound and Young's inequality, as we detail next. Suppress the index $k$ and identify output matrices with their row-wise vectorizations, preserving inner products and norms. Set $M=JJ^{\T}$ and $\delta=h-g$. Then
\[
 M\succeq\lambda I,\qquad \|M\|_{\mathrm{op}}\leq B_1^2,
 \qquad \|\delta\|_{\F}\leq B_3\|e\|_{\F}.
\]
The margin and Lemma~\ref{l41} give $\|\theta^+-\theta\|\leq\eta B_1(B_0+B_3r)\leq R$ for every batch. Thus the entire parameter segment lies in $Q_{\mathrm{out}}$, and the output endpoints and their chord lie in $D$. Lemma~\ref{l42} applies even if the next state exits the inner region or its cache error exceeds $r$.

Since $h$ is formed before sampling, all $p_i$ are $\mathcal{F}_k$-measurable. Uniform size-$b$ sampling gives
\[
 \E[H^{\T}\mid\mathcal{F}_k]=p\sum_{i=1}^N p_i=pJ^{\T}h,
 \qquad \sum_{i=1}^N z_i=J^{\T}g.
\]
Subtract $F^*$ in Lemma~\ref{l42} and take conditional expectation to obtain
\[
 \E[f^+\mid\mathcal{F}_k]\leq f-p\eta\langle g,Mh\rangle+c_2\eta^2\|h\|_{\F}^2.
\]
The full row-rank assumption and Young's inequality yield
\[
 \begin{aligned}
 \langle g,Mh\rangle
 &=\langle g,Mg\rangle+\langle g,M\delta\rangle\\
 &\geq\lambda\|g\|_{\F}^2-B_1^2B_3\|g\|_{\F}\|e\|_{\F}\\
 &\geq\frac{\lambda}{2}\|g\|_{\F}^2-\frac{B_1^4B_3^2}{2\lambda}\|e\|_{\F}^2.
 \end{aligned}
\]
This proves \eqref{eq:H_gap_descent} with $D_H=pB_1^4B_3^2/(2\lambda)$.

Apply Lemma~\ref{l43} with auxiliary scalar $s_0=2/p$. Since $q=(1+p/2)(1-p)=1-p(1+p)/2=1-d_0$,
\[
 \E[\|e^+\|_{\F}^2\mid\mathcal{F}_k]\leq(1-d_0)\|e\|_{\F}^2+C_H\eta^2\|h\|_{\F}^2.
\]
Multiplying by $\beta$ and adding \eqref{eq:H_gap_descent}, with $A_H=c_2+\beta C_H$, gives
\[
 \begin{aligned}
 \E[V^+\mid\mathcal{F}_k]
 &\leq V-\frac{p\lambda\eta}{2}\|g\|_{\F}^2-(\beta d_0-D_H\eta)\|e\|_{\F}^2+A_H\eta^2\|h\|_{\F}^2\\
 &\leq V-\left(\frac{p\lambda\eta}{2}-2A_H\eta^2\right)\|g\|_{\F}^2
       -\left(\beta d_0-D_H\eta-2A_HB_3^2\eta^2\right)\|e\|_{\F}^2.
 \end{aligned}
\]
Here consistency and $\text{A}_6$ imply $\|h\|_{\F}^2\leq2\|g\|_{\F}^2+2B_3^2\|e\|_{\F}^2$. We check both coefficients: $\eta\leq p\lambda/(8A_H)$ gives
\[
 \frac{p\lambda\eta}{2}-2A_H\eta^2\geq\frac{p\lambda\eta}{4}=a_H\eta.
\]
Also $\beta d_0=4(D_H+1)$ and $\eta\leq1$ give $D_H\eta\leq D_H\leq\beta d_0/4$. Combining this with $2A_HB_3^2\eta^2\leq\beta d_0/4$ gives
\[
 \beta d_0-D_H\eta-2A_HB_3^2\eta^2\geq\frac{\beta d_0}{2}=\kappa\beta.
\]
Substitution proves the result. No division by $B_3$ or $1-p$ was used, so $B_3=0$ and $p=1$ are covered.
\end{proof}
\subsection{Proof of Lemma \ref{l45}.}
\label{d_l45}
\begin{proof}
The state $(\theta_k,\tilde{Y}_k)$ is $\mathcal{F}_k$-measurable, so $\tau_H$ is a stopping time. Exact initialization gives $e_0=0$ and $V_0=f(Y_0)$, and the strict initial inequality implies $\tau_H>0$. On $\{k<\tau_H\}$,
\[
 \theta_k\in Q_{\mathrm{in}},\qquad
 \beta\|e_k\|_{\F}^2\leq V_k<\bar v\leq\beta r^2.
\]
Thus Lemma~\ref{l44} applies at every active step. Importantly, it applies to the update reaching $\tau_H$ as well; no assumption at the exit cache state is required.

We verify boundedness and integrability at a finite exit $j=\tau_H$. The preceding state satisfies the displayed bounds, and Lemma~\ref{l41} with the margin gives
\[
 \|\theta_j-\theta_{j-1}\|\leq\eta B_1(B_0+B_3r)\leq R,
 \qquad \theta_j\in Q_{\mathrm{out}}.
\]
The parameter segment also lies in $Q_{\mathrm{out}}$, and hence
\[
 \|Y_j-Y_{j-1}\|_{\F}\leq B_1\|\theta_j-\theta_{j-1}\|
 \leq B_1^2\eta(B_0+B_3r).
\]
Let $P_{\mathcal{B}_{j-1}^c}$ denote the row projection onto the unrefreshed rows. The unchanged cache formula gives the exact identity
\[
 e_j=P_{\mathcal{B}_{j-1}^c}(e_{j-1}+Y_{j-1}-Y_j).
\]
Since a row projection has operator norm at most one,
\[
 \|e_j\|_{\F}\leq\|e_{j-1}\|_{\F}+\|Y_j-Y_{j-1}\|_{\F}
 \leq r+B_1^2\eta(B_0+B_3r).
\]
The compact set $Q_{\mathrm{out}}$ and continuity of $f\circ Y$ give the finite deterministic constant
\[
 M_f=\max_{\theta\in Q_{\mathrm{out}}}f(Y(\theta)).
\]
Before exit $0\leq V_k<\bar v$; at exit and thereafter in the frozen process,
\[
 0\leq U_k\leq M_f+\beta\bigl[r+B_1^2\eta(B_0+B_3r)\bigr]^2.
\]
Taking the maximum of this bound and $\bar v$ gives a uniform deterministic bound for all $U_k$. Nonnegativity at exit follows from $Y(\theta_j)\in D$ and $f\geq0$ there. This argument never evaluates the surrogate at an exit cache state whose error might exceed $r$.

On $\{k\geq\tau_H\}$ the process is frozen, so $U_{k+1}=U_k$. On $\{k<\tau_H\}$ the integer-valued stopping time satisfies $k+1\leq\tau_H$, giving $U_k=V_k$ and $U_{k+1}=V_{k+1}$. These events are $\mathcal{F}_k$-measurable. Applying Lemma~\ref{l44} on the second event gives
\[
 \E[U_{k+1}\mid\mathcal{F}_k]
 \leq U_k-\mathbf{1}_{\{k<\tau_H\}}
 \bigl(a_H\eta\|g_k\|_{\F}^2+\kappa\beta\|e_k\|_{\F}^2\bigr).
\]
The drift terms are understood as zero on the frozen event; no post-exit surrogate is needed. Adaptedness, integrability, and this inequality show that $U$ is a nonnegative bounded supermartingale. Taking expectations and iterating gives $\E U_k\leq U_0=V_0$.

For completeness, the nonnegative supermartingale maximal inequality gives, for each finite $T$ and $a>0$,
\[
 \Pr\!\left(\max_{0\leq k\leq T}U_k\geq a\right)\leq V_0/a.
\]
It follows either directly by stopping at the first crossing of $a$ before $T$, or from Doob's inequality. To include a supremum that may not be attained, apply this bound at levels $a<\bar v$, let $T\to\infty$, and then let $a\uparrow\bar v$. This proves $\Pr(\sup_{k\geq0}U_k\geq\bar v)\leq V_0/\bar v$.
\end{proof}
\subsection{Proof of Lemma \ref{l46}.}
\label{d_l46}
\begin{proof}
We first relate the KL inequality for the objective to the coupled Lyapunov function. Before exit, $f_k\geq0$ and $\|e_k\|_{\F}\leq r$. Subadditivity of the power function for $0<\alpha<1$ gives
\[
 \begin{aligned}
 V_k^\alpha
 &=(f_k+\beta\|e_k\|_{\F}^2)^\alpha
 \leq f_k^\alpha+\beta^\alpha\|e_k\|_{\F}^{2\alpha}\\
 &\leq c^{-1}\|g_k\|_{\F}+\beta^\alpha r^{2\alpha-1}\|e_k\|_{\F}
 \leq K_\alpha W_k.
 \end{aligned}
\]
Here $2\alpha-1\geq0$, so the cache estimate is valid also at $\alpha=1/2$ and at $e_k=0$. The definition of $m_H$ and $\eta\leq1$ imply
\[
 a_H\eta\|g_k\|_{\F}^2+\kappa\beta\|e_k\|_{\F}^2
 \geq m_H\eta(\|g_k\|_{\F}^2+\|e_k\|_{\F}^2)=m_H\eta W_k^2.
\]
Consequently Lemma~\ref{l45} yields
\begin{equation}\label{eq:H_W_drift}
 U_k-\E[U_{k+1}\mid\mathcal{F}_k]
 \geq m_H\eta\mathbf{1}_{\{k<\tau_H\}}W_k^2.
\end{equation}

The function $\Psi(t)=K_\alpha t^{1-\alpha}/(1-\alpha)$ is nonnegative, increasing, and concave on $[0,\infty)$, with derivative $\Psi'(t)=K_\alpha t^{-\alpha}$ for $t>0$. At an active state with $V_k>0$, its supporting-line inequality applies to every nonnegative $U_{k+1}$, including an exit value larger than $\bar v$:
\[
 \Psi(U_{k+1})\leq\Psi(V_k)+\Psi'(V_k)(U_{k+1}-V_k).
\]
Taking conditional expectation, using \eqref{eq:H_W_drift} and $V_k^\alpha\leq K_\alpha W_k$, gives
\[
 \begin{aligned}
 \Psi(U_k)-\E[\Psi(U_{k+1})\mid\mathcal{F}_k]
 &\geq m_H\eta K_\alpha V_k^{-\alpha}W_k^2\\
 &\geq m_H\eta W_k.
 \end{aligned}
\]
The conditional expectations are finite because $U$ is uniformly bounded. Formally, one can first condition the supporting-line inequality on the measurable sets $\{k<\tau_H,\ V_k\geq1/j\}$, where its derivative is bounded, and then take their union over $j$. No integrability assumption on $V_k^{-\alpha}$ is needed.

At zero we argue separately. If $V_k=0$ before exit, then $f_k=0$ and $e_k=0$. Since $D$ is open, $F$ is differentiable there, and $f\geq0$, the point $Y_k\in D$ is a minimum and $g_k=0$. Consistency gives $h_k=g_k=0$. Every batch therefore has $H_k=0$, the parameter does not move, and the exact cache remains exact after refresh. Thus $V_{k+1}=W_k=0$. On $\{k\geq\tau_H\}$, $U$ is frozen. Combining the cases gives
\[
 \E[\Psi(U_{k+1})\mid\mathcal{F}_k]
 \leq\Psi(U_k)-m_H\eta\mathbf{1}_{\{k<\tau_H\}}W_k
\]
without differentiating $\Psi$ at zero.

Take expectations and sum from $k=0$ to $T$. Since $U_0=V_0$ is deterministic and $\Psi\geq0$,
\[
 m_H\E\sum_{k=0}^T\eta\mathbf{1}_{\{k<\tau_H\}}W_k
 \leq\Psi(V_0)-\E\Psi(U_{T+1})\leq\Psi(V_0).
\]
Monotone convergence for the nonnegative sums proves
\[
 \E\sum_{k<\tau_H}\eta W_k\leq\frac{\Psi(V_0)}{m_H}.
\]
At every active step, Lemma~\ref{l41}, consistency, and Cauchy--Schwarz in $\mathbb{R}^2$ give
\[
 \begin{aligned}
 \|\theta_{k+1}-\theta_k\|
 &=\eta\|H_k\|_{\F}\leq\eta B_1\|h_k\|_{\F}\\
 &\leq\eta B_1(\|g_k\|_{\F}+B_3\|e_k\|_{\F})
 \leq\eta B_1\sqrt{1+B_3^2}\,W_k.
 \end{aligned}
\]
Thus the stopped total length $L_H=\sum_{k<\tau_H}\|\theta_{k+1}-\theta_k\|$ satisfies
\[
 \E L_H\leq\frac{B_1\sqrt{1+B_3^2}\,\Psi(V_0)}{m_H},
\]
which also implies $L_H<\infty$ almost surely.

If a finite exit satisfies $\|\theta_{\tau_H}-\theta_0\|\geq R$, the triangle inequality gives
\[
 R\leq\sum_{k=0}^{\tau_H-1}\|\theta_{k+1}-\theta_k\|\leq L_H.
\]
The last step is included because its starting index is strictly less than $\tau_H$. If the Lyapunov threshold triggers exit, $U_{\tau_H}=V_{\tau_H}\geq\bar v$. These possibilities cover every finite exit, including simultaneous crossings. A union bound, Markov's inequality, and Lemma~\ref{l45} give
\[
 \begin{aligned}
 \Pr(\tau_H<\infty)
 &\leq\Pr(L_H\geq R)+\Pr(\sup_{k\geq0}U_k\geq\bar v)\\
 &\leq\frac{B_1\sqrt{1+B_3^2}\,\Psi(V_0)}{m_H R}+\frac{V_0}{\bar v}.
 \end{aligned}
\]
Every estimate concerns active steps or the single exit value, not the original trajectory after exit.
\end{proof}

\subsection{Proof of Theorem \ref{t47}.}
\label{d_t47}
\begin{proof}
The probability bound is Lemma~\ref{l46}. Write $Z_k=U_k\mathbf{1}_{\{k<\tau_H\}}$. Since survival at time $k+1$ implies survival at time $k$ and $U_{k+1}\geq0$, Lemmas~\ref{l44}--\ref{l46} give
\[
 \begin{aligned}
 \E[Z_{k+1}\mid\mathcal{F}_k]
 &\leq\mathbf{1}_{\{k<\tau_H\}}\E[U_{k+1}\mid\mathcal{F}_k]\\
 &\leq Z_k-m_H\eta\mathbf{1}_{\{k<\tau_H\}}W_k^2\\
 &\leq Z_k-d_H Z_k^{2\alpha}.
 \end{aligned}
\]
Jensen's inequality for the nonnegative killed variable $Z_k$ proves \eqref{eq:H_killed_rate}; it is applied to the killed process. Event restriction and $\Pr(E_H)\geq1-\zeta$ give the expectation bounds. For $\alpha=1/2$, iteration gives the geometric bound.

For every exponent, the recurrence makes $n_k$ decrease to zero: a positive limit would force a fixed positive decrement. Lemma~\ref{l45} gives almost sure convergence of $U_k$ to a nonnegative limit $U_\infty$. Fatou's lemma and $\E[U_k\mathbf{1}_{E_H}]\leq n_k\to0$ imply $U_\infty=0$ almost surely on $E_H$. There $U_k=V_k$, so $f_k\to0$ and $\|e_k\|_{\F}^2\leq V_k/\beta\to0$. The finite expected stopped length in Lemma~\ref{l46} implies finite length almost surely, hence on $E_H$ the parameter sequence is Cauchy. Its limit belongs to $\overline{Q_{\mathrm{in}}}\subset Q_{\mathrm{out}}$, and continuity of $F\circ Y$ gives $f(Y(\theta_\infty))=0$.
\end{proof}

\subsection{Proof of Corollary \ref{coro48}.}
\label{d_coro48}
\begin{proof}
Markov's inequality and Theorem~\ref{t47} imply
\[
 \Pr(E_H\cap\{U_k>\widehat\rho^k\})
 \leq\widehat\rho^{-k}\E[U_k\mathbf{1}_{E_H}]
 \leq V_0(\rho_H/\widehat\rho)^k.
\]
This series is summable. The Borel--Cantelli lemma therefore excludes infinitely many violations on $E_H$, where $U_k=V_k$; the two component bounds follow from the definition of $V_k$.
\end{proof}

\subsection{Proof of Theorem \ref{t49}.}
\label{d_t49}
\begin{proof}
Put $t=2\alpha-1>0$. The killed Jensen recurrence \eqref{eq:H_killed_rate} reads $n_{k+1}\leq n_k-d_H n_k^{1+t}$. Whenever $n_{k+1}>0$, convexity of $x\mapsto x^{-t}$ on $(0,\infty)$ gives
\[
 n_{k+1}^{-t}-n_k^{-t}
 \geq t n_k^{-t-1}(n_k-n_{k+1})\geq t d_H.
\]
Summing proves the asserted bound. If some $n_j=0$, nonnegativity and the recurrence give $n_k=0$ for $k\geq j$, so the bound remains valid. The event-restricted bounds follow from Theorem~\ref{t47}. If $V_0=0$, nonnegativity of $f$ on the open set $D$ implies $g_0=0$, while exact initialization gives $e_0=0$ and consistency gives $h_0=0$. Both updates are then fixed, and induction proves the zero case.
\end{proof}

\subsection{Proof of Corollary \ref{coro410}.}
\label{d_coro410}
\begin{proof}
The almost sure convergence assertion follows from Theorem~\ref{t47}. Let $\gamma=(2\alpha-1)^{-1}$ and define the killed process
\[
X_k^H=U_k\mathbf1_{\{k<\tau_H\}}.
\]
The stopped conditional descent in Lemma~\ref{l45} makes $(X_k^H)$ a nonnegative supermartingale, and Theorem~\ref{t49} gives
\[
\E X_k^H=O(k^{-\gamma}).
\]
Fatou's lemma implies that $\liminf_k k^\gamma X_k^H$ is finite almost surely. On $E_H$, Theorem~\ref{t47} gives $V_k\to0$ and $X_k^H=V_k$. Enumerating infinitely many indices at which $k^\gamma V_k$ is bounded proves the stated random subsequence.

Fix $\epsilon\in(0,\gamma)$ and put $a=\gamma-\epsilon>0$. For every integer $j\geq1$, the finite-horizon maximal inequality for the nonnegative supermartingale $(X_k^H)_{k\geq2^j}$, followed by monotone convergence as the horizon tends to infinity, gives
\[
\begin{aligned}
\Pr\left(\sup_{k\geq2^j}X_k^H>2^{-ja}\right)
&\leq 2^{ja}\E X_{2^j}^H
\leq C2^{-j\epsilon}
\end{aligned}
\]
for a deterministic constant $C$. The right-hand side is summable in $j$, so Borel--Cantelli implies that, almost surely, for all sufficiently large $j$,
\[
\sup_{k\geq2^j}X_k^H\leq2^{-ja}.
\]
For $2^j\leq k<2^{j+1}$, this gives $X_k^H\leq2^a k^{-a}$. Hence
\[
X_k^H=O_\omega\!\left(k^{-\gamma+\epsilon}\right)\quad\text{almost surely}.
\]
On $E_H$, $X_k^H=V_k$, and $0\leq f_k\leq V_k$, while $V_k\geq\beta\|e_k\|_{\F}^2$. Therefore
\[
V_k=O_\omega\!\left(k^{-\gamma+\epsilon}\right),\qquad
\|e_k\|_{\F}=O_\omega\!\left(k^{-(\gamma-\epsilon)/2}\right)
\quad\text{almost surely on }E_H.
\]
All conclusions are restricted to $E_H$.
\end{proof}

\end{document}